\documentclass[12pt,leqno,a4paper]{amsart}
\usepackage{enumerate,cite}
\usepackage{amssymb}
\usepackage{amsthm}      
\usepackage{amsmath}
\usepackage{txfonts}
\usepackage{mathrsfs}
\usepackage{hyperref}
\usepackage{color}
\usepackage[all]{xy}

\newcommand{\F}{\mathbb{F}}

\newcommand{\G}{\mathrm{G}}
\newcommand{\R}{\mathrm{R}}

\newcommand{\bG}{\mathbf{G}}
\newcommand{\bX}{\mathbf{X}}
\newcommand{\bP}{\mathbf{P}}
\newcommand{\bH}{\mathbf{H}}

\newcommand{\bL}{\mathbf{L}}

\newcommand{\bY}{\mathbf{Y}}
\newcommand{\bU}{\mathbf{U}}
\newcommand{\bBr}{\mathbf{Br}}
\newcommand{\bNBr}{\mathbf{NBr}}

\newcommand{\cB}{\mathcal{B}}
\newcommand{\cA}{\mathcal{A}}

\newcommand{\cF}{\mathcal{F}}

\newcommand{\cO}{\mathcal{O}}
\newcommand{\cP}{\mathcal{P}}
\newcommand{\cC}{\mathcal{C}}

\newcommand{\fS}{\mathfrak{S}}

\newcommand{\mrO}{\mathrm{O}}
\newcommand{\Aut}{\operatorname{Aut}\nolimits}
\newcommand{\End}{\operatorname{End}\nolimits}
\newcommand{\Alp}{\operatorname{Alp}\nolimits}
\newcommand{\IBr}{\operatorname{IBr}\nolimits}
\newcommand{\Ind}{\operatorname{Ind}}
\newcommand{\Lin}{\operatorname{Lin}}
\newcommand{\Irr}{\operatorname{Irr}\nolimits}
\newcommand{\Out}{\operatorname{Out}\nolimits}
\newcommand{\Res}{\operatorname{Res}}
\newcommand{\GL}{\operatorname{GL}}

\newcommand{\SL}{\operatorname{SL}}
\newcommand{\SU}{\operatorname{SU}}
\newcommand{\Sp}{\operatorname{Sp}}
\newcommand{\Spin}{\operatorname{Spin}}

\newcommand{\PSp}{\operatorname{PSp}}
\newcommand{\dz}{\operatorname{dz}}
\newcommand{\rdz}{\operatorname{rdz}}

\newcommand{\Bl}{\operatorname{Bl}}
\newcommand{\Br}{\operatorname{Br}}
\newcommand{\bl}{\operatorname{bl}}
\newcommand{\Char}{\operatorname{Char}}
\newcommand{\BrCh}{\operatorname{BrCh}}
\newcommand{\opp}{\operatorname{opp}}
\newcommand{\N}{\operatorname{N}}
\newcommand{\C}{\operatorname{C}}
\newcommand{\Tr}{\operatorname{Tr}}

\newcommand{\tpsi}{\widetilde{\psi}}
\newcommand{\hpsi}{\widehat{\psi}}
\newcommand{\tQ}{\widetilde{Q}}

\newcommand{\tvhi}{\widetilde{\vhi}}
\newcommand{\hvhi}{\widehat{\vhi}}
\newcommand{\ts}{\tilde{s}}
\newcommand{\tB}{\widetilde{B}}
\newcommand{\tG}{\widetilde{G}}

\newcommand{\tM}{\widetilde{M}}
\newcommand{\hG}{\widehat{G}}
\newcommand{\hN}{\widehat{N}}
\newcommand{\hM}{\widehat{M}}

\newcommand{\hQ}{\widehat{Q}}
\newcommand{\tL}{\widetilde{L}}

\newcommand{\tcB}{\widetilde{\cB}}
\newcommand{\tbG}{\widetilde{\mathbf{G}}}
\newcommand{\tbL}{\widetilde{\mathbf{L}}}

\newcommand{\wOm}{\widetilde{\Omega}}

\let\la=\lambda
\let\vhi=\varphi
\let\Ga=\Gamma

\let\ti=\times

\theoremstyle{theorem}
\newtheorem{mainthm}{Theorem}
\newtheorem{thm}{Theorem}[section]
\newtheorem{lem}[thm]{Lemma}
\newtheorem{prop}[thm]{Proposition}
\newtheorem{cor}[thm]{Corollary}
\newtheorem{amp}[thm]{Assumption}
\newtheorem{hpo}[thm]{Hypothesis}

\theoremstyle{definition}
\newtheorem{defn}[thm]{Definition}
\newtheorem{rmk}[thm]{Remark}

\numberwithin{equation}{section}

\begin{document}

\title[Jordan decomposition for Alperin weight conjecture]{Jordan decomposition for weights and the blockwise\\ Alperin weight conjecture}

\author{Zhicheng Feng}
\address{School of Mathematics and Physics, University of Science and Technology Beijing, Beijing 100083, China}
\email{zfeng@pku.edu.cn}

\author{Zhenye Li}
\address{College of Mathematics and Physics, Beijing University of Chemical Technology, Beijing 100029, China}
\email{lizhenye@pku.edu.cn}

\author{Jiping Zhang}
\address{SICM, Southern University of Science and Technology, Shenzhen 518055, China;  School of Mathematical Sciences, Peking University, Beijing 100871, China.}
\email{jzhang@pku.edu.cn}

\thanks{Supported by the NSFC (National Natural Science Foundation of China) No. 11901028~(Feng), No. 12001032 (Li)~and No. 11631001~(Zhang).}

\begin{abstract}
The Alperin weight conjecture was reduced to simple groups  by the work of Navarro--Tiep and Sp\"ath.
To prove Alperin weight conjecture, it suffices to show that all finite non-abelian simple groups are BAW-good.
We reduce the verification of the inductive conditons for groups of Lie type in non-defining characteristic to quasi-isolated blocks.
\end{abstract}




\maketitle


\section{Introduction}

The local-global conjectures in the representation theory  of finite groups predict global data in terms of local information.
Alperin's weight conjecture is one, which asserts that the number of irreducible Brauer
characters of a block of a finite group may be determined by counting the number
of conjugacy classes of the so-called weights for this block.

In the domain of representation theory, some significant breakthroughs have been achieved using the classification of finite simple groups. 
For example, Isaacs, Malle and Navarro \cite{IMN07} reduced the McKay conjecture to simple groups in 2007 and using this reduction,
Malle and Sp\"ath \cite{MS16} finally proved the McKay conjecture for all finite groups at the prime 2.

In 2011, the non-blockwise version of Alperin's weight conjecture was  reduced to a verification in finite simple groups by Navarro and Tiep \cite{NT11}.
Soon after, Sp\"ath \cite{Sp13} extended this reduction to the blockwise setting:
if every finite non-abelian simple group satisfies the so-called \emph{inductive blockwise Alperin weight (BAW)} condition, then the blockwise Alperin weight conjecture holds for every block of every finite group.
If a finite  non-abelian simple group $S$ satisfies the inductive BAW condition, then we also say that $S$ is \emph{BAW-good}.
The inductive BAW condition has been verified for simple alternating groups, Suzuki and Ree groups by Malle \cite{Ma14},
for groups of Lie type in their defining characteristic by Sp\"ath \cite{Sp13},
for groups of types $\mathsf G_2$ and ${}^3\mathsf D_4$ by Schulte \cite{Sch16} and  for certain cases of classical groups by several authors \cite{FLZ20a,FLZ20b,FLZ20c,FM20,Li21}.

The 16 infinite families of simple groups of Lie type form the bulk of the finite simple groups, so understanding the representations of finite groups of Lie type is important to proving the inductive conditions for simple groups.
Is there an analogue of Jordan decomposition for weights of finite groups of Lie type? Malle proposed this problem in Problem 4.9 of\cite{Ma17}.
Following Kessar–Malle \cite[p.~101]{KM19}, we hope for a Bonnaf\'e--Rouquier type reduction (cf. \cite{BR03}) to a few special situations (\emph{i.e.}, quasi-isolated blocks).

Recently, Ruhstorfer \cite{Ru20b} reduced the verification of the inductive condition of the Alperin--McKay conjecture for groups of Lie type in non-defining characteristic to quasi-isolated blocks.
Inspired by this success, in this paper, we give positive answers for the questions of Malle and Kessar--Malle as in the above paragraph and 
establish the following theorem which gives a reduction of the verification for the inductive BAW condition to quasi-isolated blocks.

\begin{mainthm}\label{mian-thm}
Assume that the quasi-isolated $\ell$-blocks of every finite quasi-simple group of Lie type over a field of characteristic different from $\ell$ satisfy the inductive blockwise Alperin weight condition and Assumption \ref{assumption:stabilizer-extend} holds  (in the sense of Hypothesis \ref{Hypo-quasi-isolated}).
Let $S$ be a simple group of Lie type with non-exceptional Schur multiplier defined over  a field of characteristic $p$ different from $\ell$.
Then $S$ is BAW-good at the prime $\ell$. 
\end{mainthm}

Assumption \ref{assumption:stabilizer-extend} predicts the stabilizer and the extendibility of  irreducible Brauer characters of quasi-simple groups, and it is one of the main ingredients of the criterion of the inductive BAW condition given by Brough--Sp\"ath \cite{BS20}.
Note that a similar result for ordinary characters, called the $A(\infty)$ condition, was established by Cabanes and Sp\"ath during the investigation of the inductive condition of the McKay conjecture; see Remark \ref{rmk:untri-basic-set}.
In this way, Assumption \ref{assumption:stabilizer-extend} holds if the unitriangularity of the decomposition matrix is known.
For example, we mention that the unipotent characters form unitriangular basic sets at good primes by  recent work of Brunat, Dudas and Taylor \cite{BDT20}.

Let $\bG$ be a connected reductive algebraic group over an algebraic closure of the field $\F_p$ with $p$ elements and $F:\bG\to\bG$ a Steinberg map.
Suppose that $(\bG^*,F^*)$  dual to $(\bG,F)$.
Let $\ell$ be a prime different from $p$ and let $(K,\cO,k)$ be a splitting $\ell$-modular system for $\bG^F$.
For a semisimple $\ell'$-element
$s\in{\bG^*}^{F^*}$,
by a theorem of Brou\'e--Michel \cite{BM89}, we associate to the ${\bG^*}^{F^*}$-conjugacy class containing $s$ a central idempotent  $e_s^{\bG^F}$ of $\cO \bG^F$.
Assume that $\C_{{\bG^*}^{F^*}}(s)\C_{\bG^*}^\circ(s)$ is contained in an $F^*$-stable Levi subgroup $\bL^*$ of $\bG^*$ and $\bL\le \bG$ is dual to $\bL^*$.

In 2003, Bonnaf\'e and Rouquier \cite{BR03} showed that Lusztig induction induces a Morita equivalence between $\cO \bL^F e_s^{\bL^F}$ and $\cO \bG^F e_s^{\bG^F}$.
This is a crucial ingredient in the proof of the ``if" direction of the Brauer Height Zero Conjecture by Kessar--Malle \cite{KM13}.
The Bonnaf\'e--Rouquier equivalence was improved by Bonnaf\'e--Dat--Rouquier \cite{BDR17}: $\cO \bL^F e_s^{\bL^F}$ and $\cO \bG^F e_s^{\bG^F}$ are splendid Rickard equivalent, and therefore the Brauer categories of these two block algebras are equivalent.
In \cite{Ru20a}, Ruhstorfer  obtained a local version and extended the Bonnaf\'e--Rouquier equivalence by including automorphism groups.

For an $\ell$-subgroup $Q$ of $\bG^F$, we extend the result of Ruhstorfer \cite{Ru20a} on local subgroups to quotient groups and prove a Morita equivalence between block algebras of $\N_{\bL^F}(Q)/Q$ and $\N_{\bG^F}(Q)/Q$.
From this we obtain a natural correspondence of conjugacy classes of weights in $e_s^{\bL^F}$ and $e_s^{\bG^F}$, which is equivariant under the automorphisms.
In this way, we prove a Jordan decompsition of the inductive BAW condition and reduce to quasi-isolated blocks.

The quasi-isolated semisimple elements were classified by Bonnaf\'e \cite{Bo05} and the quasi-isolated blocks are better understood by the work Cabanes, Enguehard, Kessar and Malle; see \cite{KM13} and \cite[\S 9]{Cr19} for a precise historical account.
In particular, the unipotent blocks are an important class of quasi-isolated blocks. 
We mention that the unipotent blocks had been checked to satisfy the inductive BAW condition for many cases, especially for most of the classical groups \cite{FLZ19,FLZ20b}; see Remark \ref{rmk:checked-case-uni}.

We hope that Theorem  \ref{mian-thm} can provide significant simplifications in verifying the inductive BAW condition.
For example, it is useful in the final proof of the verification  for  type $\mathsf A$; see \cite{FLZ21}.

\vspace{2ex}

This paper is structured as follows.
After introducing some notation in Section \ref{sec:pre}, 
we recall the notion of block isomorphism of modular triples and covering of weights, and formulate the criterion of the inductive BAW condition given by Brough--Sp\"ath in
Section \ref{sec:char-triple}.
A Morita equivalence of block algebras of quotient groups is obtained in Section \ref{sec:lie-type-equivalence} and from this we
derive a Jordan decomposition of weights of finite groups of Lie type in non-defining characteristic.
Afterwards, we establish a Jordan decomposition for the inductive BAW condition and prove Theorem \ref{mian-thm} in Section \ref{sec:red-quasi-isolated}.

\vskip 1pc
\noindent{\bf Acknowledgement:} 
We thank Gunter Malle and Lucas Ruhstorfer for their helpful remarks on an earlier version and thank the anonymous referee for useful comments and suggestions.
We also thank Gabriel Navarro for his comment.


\section{Preliminaries}\label{sec:pre}

In this section, we introduce some notation and background material from the representation theory of finite groups. 
Throughout, we fix a prime number $\ell$ and consider modular representations with respect to  $\ell$.

\subsection{General notation}
Let $G$ be a finite group. 
Concerning the $\ell$-blocks and ($\ell$-Brauer) characters of $G$, we mainly follow the notation of \cite{Na98}.
We denote the restriction and induction by $\Res$ and $\Ind$ respectively.
For $N\unlhd G$ we sometimes identify the (Brauer) characters of $G/N$ with their inflations to $G$.
If $\theta\in \Irr(N)$ (resp. $\theta\in\IBr(N)$), then we denote by $\Char(G\mid \theta)$ (resp. $\BrCh(G\mid \theta)$) the set of possibly reducible (Brauer) characters $\chi$ of $G$ such that $\Res^G_N\chi$ is a multiple of $\theta$.
If $\chi\in\Irr(G)\cup\IBr(G)$, we denote by $\bl(\chi)$ the block of $G$ containing $\chi$.
Sometimes we also write $\bl(\chi)$ as $\bl_{G}(\chi)$ to indicate the group $G$.
For $\chi\in\Irr(G)$, we denote by $\chi^0$ its restriction to the $\ell$-regular elements of $G$.

As usual, the cardinality of a set, or the order of a finite group, $X$, is denoted by $|X|$.
If a group $A$ acts on a finite set $X$, we denote by $A_x$ the stabilizer of $x\in X$ in $A$, analogously we denote by $A_Y$ the setwise stabilizer of $Y\subseteq X$.
If $A$ acts on a finite group $G$ by automorphisms, then  $A$ acts on $\Irr(G)\cup\IBr(G)$ naturally, where ${}^{a^{-1}}\chi(g)=\chi^a(g)=\chi(g^{a^{-1}})$ for every $g\in G$, $a\in A$ and $\chi\in\Irr(G)\cup\IBr(G)$.
If $H\le G$ and $\chi\in\Irr(G)$, we denote by $A_{H,\chi}$ the stabilizer of $\chi$ in $A_H$.

If moreover $G$ is abelian,
then all irreducible characters of $G$ are linear characters.
We also write $\Lin(G)$ for $\Irr(G)$.
The group $\Lin(G)$ is isomorphic to $G$.
In addition, $\Lin_{\ell'}(G)$ denotes the subset of $\Lin(G)$ consisting of the elements of $\ell'$-order so that $\Lin_{\ell'}(G)$ and $\mrO_{\ell'}(G)$ are isomorphic.

For a finite group $G$, we let $\dz(G)$ denote the set of all irreducible $\ell$-defect zero characters of $G$.
Set $\dz(G \mid \theta) = \dz(G) \cap \Irr(G \mid \theta)$ for $K \le G$ and $\theta\in\Irr(K)$.
In addition, if $N\unlhd G$ and $\theta\in\Irr(N)$, then we set $$\rdz(G\mid\theta):=\{\chi\in\Irr(G\mid\theta)\mid (\chi(1)/\theta(1))_\ell=|G/N|_\ell\}.$$

A \emph{weight} of a finite group $G$ is a pair $(Q,\vhi)$, where $Q$ is an $\ell$-subgroup of $G$ and $\vhi\in\dz(\N_G(Q)/Q)$.
Note that $Q$ is necessarily a \emph{radical subgroup} of $G$, \emph{i.e.} $Q=\mrO_\ell(\N_G(Q))$.
We also call $\vhi$ a \emph{weight character}.
As usual, we often identify characters in $\dz(\N_G(Q)/Q)$ with their inflations to $\N_G(Q)$.
Let $\Alp^0(G)$ be the set of all the weight of $G$.
The group $G$ acts by conjugation on $\Alp^0(G)$ by 
$(Q,\vhi)^g=(Q^g,\vhi^g)$ for $(Q,\vhi)\in\Alp^0(G)$ and $g\in G$.
We denote by $\Alp(G)$
the set of all $G$-conjugacy classes of  weights of $G$.
For a weight $(Q,\varphi)$ of $G$, denote by $\overline{(Q,\varphi)}$ the $G$-conjugacy class containing $(Q,\varphi)$.
Sometimes we also write $\overline{(Q,\varphi)}$ simply as $(Q,\varphi)$ when no confusion can arise.

Each weight  may be assigned to a unique block.
Let $B$ be an $\ell$-block of $G$, and $(Q,\varphi)$ a weight of $G$.
We say $(Q,\varphi)$  is a $B$-weight of $G$ if $\bl_{\N_{G}(Q)}(\vhi)^G=B$.
We denote the set of $B$-weights of $G$ by $\Alp^0(B)$ and the set of conjugacy classes of $B$-weights of $G$ by $\Alp(B)$.
Then the  blockwise Alperin weight conjecture \cite{Al87}  asserts that 
$$|\IBr(B)|=|\Alp(B)|\ \ \textrm{for every block $B$ of $G$}.$$

Let $G\unlhd \tG$ and $B$ a union of $\tG$-orbits of blocks of $G$.
We denote by 
$\IBr(\tG\mid B)$ the irreducible Brauer characters of $\tG$ in the blocks of $\tG$ covering the blocks in $B$,
and by $\Alp(\tG\mid B)$ the set of conjugacy classes of weights of $\tG$ belonging to the blocks of $\tG$ covering the blocks in $B$.

\subsection{Modular representations}\label{subsec:mod-rep}

We recall some notation of local representation theory from \cite{Lin18}.
Let $(K,\cO,k)$ be an $\ell$-modular system, where $\cO$ is a complete discrete valuation ring with a residue field $k=\cO/J(\cO)$ of characteristic $\ell$ and a quotient field $K$ of characteristic zero.
In this paper, we always assume that $k$ is algebraically closed and this $\ell$-modular system is splitting, that is, if we consider a finite group $G$, then $K$ contains all roots of unity whose order divides the exponent of the group $G$.
We will use $\Lambda$ to interchangeably denote $\cO$ or $k$.

Let $A$ be a $\Lambda$-algebra which is finitely generated and projective as a $\Lambda$-module.
Denote by $A^{\opp}$ its opposite  algebra.
We write $A$-mod for the category of left A-modules that are finitely generated  as $\Lambda$-modules.

Let $G$ be a finite group.
$\Lambda G$ denotes the group algebra.
Let $\sigma\in\Aut(G)$, $H\le G$ and $M$ a left (resp. right)  $\Lambda H$-module.
Then we denote by $^\sigma M$ (resp. $M^\sigma$) the left (resp. right) $\Lambda\sigma(H)$-module, which coincides with $M$ as $\Lambda$-modules and the action of $\sigma(H)$ given by $\sigma(h)m:=hm$
(resp. by $m\sigma(h):=mh$).

Let $H$ be a subgroup of $G$.
For a $\Lambda G$-module $M$, we denote by $M^H$ the subset of $H$-fixed points in $M$.
Suppose that $H\le L$ are subgroups of $G$.
We recall that the relative trace map $\Tr^{L}_{H}:M^H\to M^L$ is defined by
$$\Tr^{L}_{H}(m)=\sum_{x\in[L/H]}x\cdot m$$ for $m\in M^H$,
where $[L/H]$ denotes a complete set of representatives for the $H$-cosets in $L$.
Let $Q$ be an $\ell$-subgroup of $G$.
We consider the Brauer functor (see, \emph{e.g.}, \cite[\S5.4]{Lin18}) $$\Br_Q^G:\Lambda G\textrm{-mod} \to k \N_G(Q)/Q\textrm{-mod},$$ which for a $\Lambda G$-module $M$ is defined by $$\Br_Q^G(M)= k\otimes_\Lambda (M^Q/\sum_{P<Q}\Tr_P^Q(M^P)).$$
We also abbreviate $\Br_Q^G$ simply as $\Br_Q$ when no confusion can arise.

Let $M$ be a $\Lambda G$-module.
Then $M$ is called an \emph{$\ell$-permutation module} if for any $\ell$-subgroup $Q$ of $G$ the restriction $\Res^{G}_{Q}(M)$ is a permutation $\Lambda Q$-module, \emph{i.e.}, $M$ is $\Lambda$-free having a $\Lambda$-basis $X$ which is permuted by the action of elements of $Q$ on $M$.
The importance of $\ell$-permutation modules always relies on their good behavior under the action of the Brauer functor; see for instance \cite[\S5.11]{Lin18}.
Denote by $\Lambda G$-perm the full subcategory of $\Lambda G$-mod whose objects are the $\ell$-permutation $\Lambda G$-modules.

We denote by $\mathrm{br}_Q: (\Lambda G)^Q\to k\C_G(Q)$ the algebra morphism given by $$\mathrm{br}_Q(\sum_{g\in G}\la_g g)=\sum_{g\in\C_G(Q)}\la_g g,$$
where $\la_g\in\Lambda$ for $g\in G$.
If $M$ is an $\ell$-permutation $\Lambda G$-module and $e\in Z(\Lambda G)$ is an idempotent, then we have $\Br_Q(Me)=\Br_Q(M)\mathrm{br}_Q(e)$.

\subsection{Brauer categories}

A \emph{Brauer pair} of $G$ (with respect to $\ell$) is a pair $(Q,b_Q)$ such that $Q$ is an $\ell$-subgroup of $G$ and $b_Q$ is an $\ell$-block of $\C_G(Q)$.
For two Brauer pairs $(Q,b_Q)$ and  $(P,b_P)$, we write
$(Q,b_Q)\le (P,b_P)$ 
and say \emph{$(Q,b_Q)$ is contained in $(P,b_P)$} 
if and only if $Q\le P$ and there is a primitive idempotent $i\in(\cO G)^P$ satisfying that $\mathrm{br}_P(i)b_P\ne 0$ and $\mathrm{br}_Q(i)b_Q\ne0$. 
Then for every Brauer pair $(P,b_P)$ and $Q\le P$, there exists a unique block $b_Q$ of $\C_G(Q)$ such that $(Q,b_Q)\le (P,b_P)$.

Let $B$ be a block of $G$.
A Brauer pair  $(Q,b_Q)$ is called a \emph{$B$-Brauer pair} if $(1,B)\le(Q,b_Q)$.
Then an $\ell$-subgroup $D$ of $G$ is a defect group of $B$ if and only if there exists a $B$-Brauer pair $(D,b_D)$ which is maximal with respect to the order relation ``$\le$".
In addition, all maximal $B$-Brauer pairs are $G$-conjugate.
For more basic facts about Brauer pairs, we refer to \cite[\S6.3]{Lin18}.

We recall the definition of Brauer category of a block as in \cite[p.~427]{Th95}.
Let $B$ be a block of $G$.
Then the \emph{Brauer category} $\bBr(G,B)$ of $B$ is the category whose objects are the $B$-Brauer pairs and the set of morphisms from  $(Q,b_Q)$ to $(P,b_P)$ consisting of all homomorphism $Q\to P$ which are given by conjugation via some $g\in G$ such that $(Q,b_Q)^g\le (P,b_P)$.

Let $(D,b_D)$ be a maximal $B$-Brauer pair.
Denote by $\cF_{(D,b_D)}(G,B)$ the full subcategory of $\bBr(G,B)$ with objects consisting of $B$-Brauer pairs contained in $(D,b_D)$.
Note that $\cF_{(D,b_D)}(G,B)$ is a fusion system on the $\ell$-group $D$; see also \cite[\S8.5]{Lin18}.
By \cite[Lemma~47.1]{Th95}, the natural inclusion functor $\cF_{(D,b_D)}(G,B)\hookrightarrow \bBr(G,B)$ induces an equivalence of categories.

Let $(Q,b_Q)$ be a Brauer pair of $G$.
We denote by $\N_{G}(Q,b_Q)$ the stabilizer of $b_Q$ in $\N_G(Q)$.
Then the idempotent $b_Q$ is also a block of $\N_{G}(Q,b_Q)$ by \cite[Exercise~40.2(b)]{Th95}.
From this, $B_Q=\Tr^{\N_G(Q)}_{\N_{G}(Q,b_Q)}(b_Q)$ the unique block of $\N_G(Q)$ covering $b_Q$.

\begin{rmk}\label{remark:wei-fusion-system}
Let $B$ be a block of a finite group $G$ and $(D,b_D)$ be a maximal $B$-Brauer pair.  
If $(Q,\vhi)$ with $Q\le D$ is a $B$-weight, then $(Q,\bl(\theta))^g\le (D,b_D)$ for some $\theta\in\Irr(\C_G(Q)\mid\vhi)$ and $g\in G$.

For any $Q\le D$, we let $(Q,b_Q)\le (D,b_D)$ be the unique $B$-Brauer pair.  
Now we enumerate the $B$-weights of $G$ from $\cF_{(D,b_D)}(G,B)$.
We also regard $\cF_{(D,b_D)}(G,B)$ as the set of its objects.
Define an equivalence relation on $\cF_{(D,b_D)}(G,B)$:
$(Q,b_Q)\sim (R,b_R)$ if and only if they are isomorphic in $\cF_{(D,b_D)}(G,B)$.
Let $(Q,b_Q)$ run over a complete set of representatives of the equivalence classes of $\cF_{(D,b_D)}(G,B)/\!\!\sim$, 
and let $\vhi$ run over the characters in $\Irr(B_Q)\cap \dz(\N_{G}(Q)/Q)$ for $B_Q=\Tr^{\N_G(Q)}_{\N_{G}(Q,b_Q)}(b_Q)$.
Then the weights $(Q,\vhi)$ obtained above form a complete set of representatives of $\Alp(B)$.
\end{rmk}

We denote by $\bNBr(G, B)$ the set of pairs $(Q,B_Q)$, where $B_Q$ is a block of $N_G(Q)$ such that $(B_Q)^G=B$.

\subsection{Equivalences of blocks of group algebras}

 Let $A$ be a $\Lambda$-algebra, finitely generated and projective as a $\Lambda$-module.
 Following \cite[2.A]{BDR17}, we denote by $\mathrm{Comp}^b(A)$ the category of bounded complexes of $A$-mod,  by
 $\mathrm{Ho}^b(A)$ the homotopy category of $\mathrm{Comp}^b(A)$,
 and by $D^b(A)$ the bounded derived category of $A$-mod.
 For $\cC\in \mathrm{Comp}^b(A)$ there is a complex $\cC^{\mathrm{red}}$ such that $\cC\simeq \cC^{\mathrm{red}}$ in $\mathrm{Ho}^b(A)$ and  $\cC^{\mathrm{red}}$ has no non-zero summand which is homotopy equivalent to $0$.
 In other words, $\cC\simeq \cC^{\mathrm{red}}\oplus \cC_0$ with $\cC_0$ homotopy equivalent to zero.
 Denote by $\End_A^\bullet(\cC)$ the total Hom-complex, with degree $n$ term $\bigoplus_{j-i=n}\mathrm{Hom}(C^i,C^j)$.
 
 Let $B$ be a $\Lambda$-algebra, finitely generated and projective as a $\Lambda$-module.
 Let $\cC$ be a bounded complex of $A$-$B$-modules, finitely generated and projective as left $A$-modules and as right $B$-modules.
 We say that $\cC$ induces a \emph{Rickard equivalence} between the algebras $A$ and $B$ if the canonical map $A\to \End^\bullet_{B^{\opp}}(\cC)^{\opp}$ is an isomorphism in $\mathrm{Ho}^b(B\otimes_\Lambda B^{\opp})$ and the canonical map  $B\to \End^\bullet_A(\cC)$ is an isomorphism in $\mathrm{Ho}^b(A\otimes_\Lambda A^{\opp})$.
 According to \cite[\S2.1]{Ri96}, $A$ and $B$ are Rickard equivalent if and only if they are derived equivalent.
 
 Let $M$ be an $A$-$B$-module.
 Then $M$ induces a Morita equivalence between $A$ and $B$ if and only if the complex $M[0]$ induces a Rickard equivalence between $A$ and $B$.
 
 Let $G$ be a finite group and $Q$ an $\ell$-subgroup of $G$.
 Since $\Br_Q$ is an additive functor,  it respects homotopy equivalences and thus extends to a functor
 $$\Br_Q^G:\mathrm{Ho}^b(\Lambda G)\to\mathrm{Ho}^b(k \N_G(Q)/Q).$$

 Let $G$ be a finite group. We denote by $G^{\opp}$ the opposite group to $G$.
 Put $$\Delta G=\{(g,g^{-1})\mid g\in G \}\subseteq G\times G^{\opp}.$$
 Let $H$ be a subgroup of $G$ and $\cC\in\mathrm{Comp}^b(\Lambda G\times (\Lambda H)^{\opp})$. 
 We say $\cC$  is  \emph{splendid} if $\cC^{\mathrm{red}}$ is a complex of $\ell$-permutation modules whose indecomposable direct summands have a vertex contained in $\Delta H$.
 Let $e\in\Lambda G$ and $f\in\Lambda H$ be central idempotents.
 If $\cC$ is splendid and induces a Rickard equivalence between $\Lambda Ge$ and $\Lambda Hf$, then we say that $\cC$ induces a \emph{splendid Rickard equivalence} between $\Lambda Ge$ and $\Lambda Hf$.
 See \cite[\S9]{Lin18} for more facts on the splendid Rickard equivalence and other equivalences between blocks of group algebras.

\section{Modular character triples and the inductive BAW condition}   \label{sec:char-triple}

Let $G$ be a finite group, $N\unlhd G$ and $\theta\in \Irr(N)\cup \IBr(N)$ with $G_\theta=G$.
Then we call $(G,N,\theta)$ a \emph{modular character triple} (or \emph{character triple}) if $\theta$ is a Brauer (or ordinary) character.
We follow the definitions and notation for the isomorphisms of character triples from \cite[Definition~(11.23)]{Is76} and of modular character triples from \cite[Definition~8.25]{Na98}.
If $(\sigma,\tau):(G,N,\theta)\to (H,M,\varphi)$ is an isomorphism of modular character triples, then $\tau$ is a group isomorphism $G/N\to H/M$ and for every $N\le J\le G$, $\sigma$ yields a bijection $\sigma_J:\BrCh(J\mid\theta)\to\BrCh(J^\tau\mid \vhi)$ with additional properties, where $\tau(J/N)=J^\tau/N$. 
The case of isomorphism of character triples is completely analogous, by replacing $\BrCh(\cdot\mid \cdot)$ to be $\Char(\cdot\mid \cdot)$.
For more basic facts on isomorphisms between (modular) character triples, see \cite[\S 11]{Is76} and \cite[\S8]{Na98}.

Furthermore, we shall also assume that all of our (modular) character triple isomorphisms are \emph{strong} in the sense of \cite[Exercise~(11.13)]{Is76} and \cite[p.~281]{SV16}.
For $N\le J\le G$, $\chi\in\IBr(J\mid\theta)$ and $\bar g=gN\in G/N$, we define $\chi^{\bar g}(x^g)=\chi(x)$. Note that this is well-defined. The modular character triple isomorphism $(\sigma,\tau):(G,N,\theta)\to (H,M,\varphi)$ is strong means $\sigma_J(\chi)^{\tau(\bar g)}=\sigma_{J^g}(\chi^{\bar g})$ for every $\bar g\in G/N$, all groups $J$ with $N\le J\le G$ and all $\chi\in\IBr(J\mid \theta)$. The situation for character triples is completely analogous.
In this section, when considering the (modular) character triples $(G,N,\theta)$ and $(H,M,\varphi)$, we always have  $G=NH$ and the considered isomorphism $\tau:G/N\to H/M$ is chosen to be the natural isomorphism.
In this way, the isomorphism between these two (modular) character triples is determined by $\sigma$ and then we abbreviate $(\sigma,\tau)$ to $\sigma$.

\subsection{Block isomorphisms between modular character triples}

Let $(G,N,\theta)$ be a modular character triple.
A projective representation $\cP:G\to\GL_{\theta(1)}(k)$ is said to be \emph{associated to} $(G,N,\theta)$ if the restriction $\cP|_N$ affords $\theta$ and $\cP(ng)=\cP(n)\cP(g)$ and $\cP(gn)=\cP(g)\cP(n)$ for all $n\in N$ and $g\in G$.

We recall our favorite isomorphisms between modular character triples from \cite{Sp17,Sp18}.
Let $(G,N,\theta)$ and $(H,M,\vhi)$ be modular character triples with $H\le G$.
We write $(G,N,\theta)\geqslant(H,M,\vhi)$ if 
\begin{itemize}
\item $G=NH$ and $M=N\cap H$,
\item there exist projective representations $\cP$ and $\cP'$ associated to $(G,N,\theta)$ and $(H,M,\vhi)$  with factor sets $\alpha$ and $\alpha'$, respectively, such that $\alpha|_{H\ti H}=\alpha'$.
\end{itemize}
Then we say that $(G,N,\theta)\geqslant(H,M,\vhi)$ is given by $(\cP,\cP')$.
By \cite[Thm.~3.1]{SV16}, one has

\begin{thm}\label{thm:triple-isomor}
Let $(G,N,\theta)\geqslant(H,M,\vhi)$ be given by $(\cP,\cP')$.
Then for every intermediate subgroup $N\le J\le G$,  the linear map \[\sigma_J:\BrCh(J\mid \theta)\to\BrCh(J\cap H\mid\vhi)\] given by
$$\mathrm{trace}(\mathcal Q\otimes \cP|_J)\mapsto \mathrm{trace}(\mathcal Q\otimes \cP'|_{J\cap H}),$$
for any projective representation $\mathcal Q$ of $J/N\cong J\cap H/M$ whose factor set is inverse to the one of $\cP|_{J}$,
is a well-defined bijection and
 induces a strong isomorphism between $(G,N,\theta)$ and $(H,M,\vhi)$. 
\end{thm}
We say that the $\sigma$ in Theorem \ref{thm:triple-isomor} is \emph{induced} by $(\cP,\cP')$.

Let $(G,N,\theta)\geqslant(H,M,\vhi)$ be given by $(\cP,\cP')$.
Then following Definition 4.19 of \cite{Sp18},
we write  $(G,N,\theta)\geqslant_b(H,M,\vhi)$ if
\begin{itemize}
	\item $\C_G(N)\le H$ and for every $c\in\C_G(N)$  the scalars associated to $\cP(c)$ and $\cP'(c)$ coincide,
	\item there is a defect group $D$ of $\bl(\vhi)$ such that $\C_G(D)\le H$, and
\item the maps $\sigma_J$ induced by $(\cP,\cP')$ satisfy $\bl(\psi)=\bl(\sigma_J(\psi))^J$ for every $N\le J\le G$ and $\psi\in\IBr(J\mid\theta)$.
\end{itemize}
Then we say that $(G,N,\theta)\geqslant_b(H,M,\vhi)$ is given by $(\cP,\cP')$.
``$\geqslant_b$" is  called the \emph{block isomorphism} between modular character triples. 

\begin{thm}[Butterfly theorem]\label{thm:butt-thm}
Let $(G_1,N,\theta)\geqslant_b(H_1,M,\vhi)$.
Suppose that $G_2$ is a finite group with normal subgroup $N$ such that $G_1/\C_{G_1}(N)$ and $G_2/\C_{G_2}(N)$ coincide as subgroups of $\Aut(N)$.
Denote by $H_2$ the subgroup of $G_2$ containing $\C_{G_2}(N)$ with $H_2/\C_{G_2}(N)=H_1/\C_{G_1}(N)$ in $\Aut(N)$.
Then 
$(G_2,N,\theta)\geqslant_b(H_2,M,\vhi)$.
\end{thm}

\begin{proof}
See \cite[Thm.~3.5]{Sp17}.	
\end{proof}	

\begin{lem}\label{lem:char-triple-iso-exten}
Let $(G,N,\theta)\geqslant_b(H,M,\vhi)$ be given by $(\cP,\cP')$ and let $\sigma$ be induced by $(\cP,\cP')$ as in Theorem \ref{thm:triple-isomor}.
Then for $N\le J\le G$  the map $$\sigma_J:\IBr(J\mid\theta)\to \IBr(J\cap H\mid \vhi)$$ is an $\N_H(J)$-equivariant bijection such that $$(\N_G(J)_\chi,J,\chi)\geqslant_b (\N_H(J)_\psi,J\cap H, \psi)$$ for every $\chi\in\IBr(J\mid \theta)$ and $\psi=\sigma_J(\chi)$.
In addition, if $J\le L\le \N_G(J)$ with abelian quotient $L/J$, then $\sigma_J$ is also $\IBr(L/J)$-equivariant.
\end{lem}

\begin{proof}
Follows by the same arguments as in \cite[Prop.~3.9 (b)]{NS14}.
\end{proof}

\begin{prop}\label{prop:indu-triple}
Let $N\unlhd G$, $H\le G$ and $M=N\cap H$ be such that $G=NH$.
Let $G_1\le G$, $H_1:=H\cap G_1$, $N_1:=N\cap G_1$ and $M_1:=M\cap G_1$ be such that $H=H_1M$ and $G=G_1N$. 
Suppose that $(G_1,N_1,\theta_1)\geqslant_b (H_1,M_1,\vhi_1)$ for $\theta_1\in\IBr(N_1)$ and $\vhi_1\in\IBr(M_1)$.
Assume that $\theta=\Ind^{N}_{N_1}(\theta_1)\in\Irr(N)$, $\vhi=\Ind^{M}_{M_1} (\vhi_1)\in\Irr(M)$, and
for every $J$ with $N\le J\le G$ induction gives a bijection between $\Irr(J\cap G_1\mid \theta_1)$ and $\Irr(J\mid \theta)$, as well as between $\Irr(J\cap H_1\mid \vhi_1)$ and $\Irr(J\cap H\mid \vhi)$.
Assume that for some defect group $D$ of $\bl(\varphi)$ one has $C_G(D)\le H$.
Then $(G,N,\theta)\geqslant_b(H,M,\vhi)$.
\end{prop}

\begin{proof}
This is completely analogous with the proof of \cite[Thm.~3.14]{NS14}.
\end{proof}

\begin{defn}\label{def:iBAW-bijection}
	Let $G$ be a finite group and $B$ a union of blocks of $G$.
	Assume that for $\Ga:=\Aut(G)_B$ we have
	\begin{enumerate}[\rm(i)]
		\item there exists a $\Ga$-equivariant bijection $\Omega:\IBr(B)\to \Alp(B)$, and 
		\item for every $\psi\in\IBr(B)$ and $\overline{(Q,\vhi)}=\Omega(\psi)$,
		$$(G\rtimes \Ga_\psi,G,\psi)\geqslant_b((G\rtimes\Ga)_{Q,\vhi},N_G(Q),\vhi^0).$$
	\end{enumerate}
	Then we say that $\Omega:\IBr(B)\to \Alp(B)$ is an \emph{iBAW-bijection} for $B$.
\end{defn}

\subsection{Extensions of characters}

First we give an elementary lemma on the extensions of  defect zero characters.

\begin{lem}\label{lem:ext-def-zero}
Let $N\unlhd G$ be finite groups such that  for every prime $r$ dividing $|G/N|$ with $r\ne\ell$, the Sylow $r$-subgroups of $G/N$ are abelian.
Let $\theta\in\dz(N)$ with $G_\theta=G$.
Then $\theta$ extends to $G$ if and only if $\theta^0$ extends to $G$.
\end{lem}
\begin{proof}
The only if direction is obvious and we only consider the if direction.
Suppose that $\theta^0$ extends to $G$ and we prove that $\theta$ extends to $G$.

We first prove for the case that $G/N$ is an $r$-group for some prime $r$.
If $r=\ell$, then this follows by \cite[Problem (3.10)]{Na98} directly.
Now we assume that $r\ne \ell$.
By the assumption, $G/N$ is abelian.
Since $\theta^0$ extends to $G$ we have
$|\IBr(G\mid \theta^0)|=|G/N|$ by \cite[Cor.~8.20]{Na98}.
Let $B$ be the union of blocks of $G$ covering $\bl(\theta)$.
Then $\IBr(B)=\IBr(G\mid \theta^0)$ and $\Irr(B)=\Irr(G\mid \theta)$.
Thus $|\Irr(G\mid \theta)|\ge |G/N|$.
Therefore, $\theta$ extends to $G$.

Now we prove the general case.
If a prime $r$ dividing $|G/N|$, we let $M_r$ be a subgroup of $G$ such that $N\le M_r$ and $M_r/N$ is a Sylow $r$-subgroup of $G/N$.
By the above arguments, $\theta$ extends to $M_r$ for every prime $r$ dividing $|G/N|$.
According to \cite[Cor.~11.31]{Is76}, $\theta$ extends to $G$.
\end{proof}

We prove the following lemma about the extendibility of induced characters.

\begin{lem}\label{lem:ext-indu}
	Let $M,N\unlhd G$, $M\le N$, $\vhi\in\Irr(M)$, $\theta\in\Irr(N)$ such that $\theta=\Ind^{N}_{M}(\vhi)$. 
	Then $\vhi$ extends to $G_\vhi$ if and only if $\theta$ extends to $G_\theta$.
	In addition, if $\tvhi\in\Irr(G_\vhi)$ is an extension of $\vhi$, then $\Ind^{G_\theta}_{G_\vhi}\tvhi$ is an extension of $\theta$.
\end{lem}

\begin{proof}
	We first claim that $G_\theta=N G_{\vhi}$.
	Obviously $N G_{\vhi}\subseteq G_\theta$.
	If $g\in G_\theta$, then $\vhi^g\in\Irr(M\mid\theta)$.
	So $\vhi^g=\vhi^n$ for some $n\in N$, and then $gn^{-1}\in G_\vhi$.
	Therefore, $g\in N G_{\vhi}$ and we have proved the claim.
	Also note that $G_\vhi\cap N=N_\vhi=M$.
	So by \cite[Cor. 4.3]{Is84}, the induction defines a bijection between $\Irr(G_\vhi\mid \vhi)$ and $\Irr(G_\theta\mid\theta)$.
	Thus $\vhi$ extends to $G_\vhi$ if and only if $\theta$ extends to $G_\theta$.	
\end{proof}

The following two lemma are the modular versions of
\cite[Cor.~4.2 and 4.3]{Is84}.

\begin{lem}\label{lem:ext-res}
Let $G=NH$ be finite groups with $N\unlhd G$ and $H\le G$ and $M=N\cap H$.
Let $\theta\in\IBr(N)$ be invariant in $G$ and assume that $\vhi=\Res^N_M\theta$ is irreducible.
Then restriction defines a bijection from $\IBr(G\mid\theta)$ to $\IBr(H\mid\vhi)$.
\end{lem}

\begin{proof}
Let $\cP$ be a projective representation associated to $(G,N,\theta)$ and $\cP'=\cP|_H$.
Then similar as \cite[Prop.~2.3]{Sp18} for Brauer characers, one has that $(\cP,\cP')$ gives $(G,N,\theta)\geqslant(H,M,\vhi)$.
So these two modular character triples are strong isomorphic.
In addition, by the construction in Theorem \ref{thm:triple-isomor}, restriction defines a bijection from $\IBr(G\mid\theta)$ to $\IBr(H\mid\vhi)$.
\end{proof}

\begin{lem}\label{lem:ext-ind}
	Let $G=NH$ be finite groups with $N\unlhd G$ and $H\le G$ and $M=N\cap H$.
	Let $\vhi\in\IBr(M)$ be invariant in $H$ and assume that $\theta=\Ind^N_M\vhi$ is irreducible.
	Then induction defines a bijection from $\IBr(H\mid\vhi)$ to $\IBr(G\mid\theta)$.
\end{lem}

\begin{proof}
Let $\cP'$ be a projective representation associated to $(H,M,\vhi)$.
We can define a projective representation $\cP$  of $G$, induced from the projective representation $\cP'$, similar as in \cite[p.~712]{NS14} for characteristic zero.
Let $n_1,\ldots, n_s\in N$ be representatives of the $M$-cosets in $N$ and
define the maps $\cP_{i,j}$ on $G$ by: $\cP_{i,j}(x)=\cP'(n_i^{-1}xn_j)$ if $n_i^{-1}xn_j\in H$, while $\cP_{i,j}(x)=0$ otherwise.
Then $\cP:G\to\GL_{\theta(1)}(k)$ is defined by
$$\cP(x)=
\left (
\begin{matrix}
\cP_{1,1}(x) & \cdots & \cP_{1,s}(x) \\
\vdots &  & \vdots \\
\cP_{s,1}(x) & \cdots & \cP_{s,s}(x)
\end{matrix}
\right  ).
$$ 
Straightforward computations show that  $\cP$ is a projective representation of $G$ associated to $\theta$ and
 the factor sets of $\cP$ and $\cP'$ coincide via the natural isomorphism $G/N\cong H/M$.
Then $(\cP,\cP')$ gives $(G,N,\theta)\geqslant(H,M,\vhi)$.
Therefore, by the construction in Theorem \ref{thm:triple-isomor},  induction defines a bijection from $\IBr(H\mid\vhi)$ to $\IBr(G\mid\theta)$.
\end{proof}

\begin{rmk}
\begin{enumerate}[\rm(i)]
	\item The proofs in Lemma \ref{lem:ext-res} and \ref{lem:ext-ind} also apply for characteristic zero situation.
	In this way, we have given an alternative proof for \cite[Cor.~4.2 and 4.3]{Is84}.
	\item Lemma \ref{lem:ext-indu} also holds if all ordinary characters are changed to be Brauer charaters.
	For the proof, just changing \cite[Cor. 4.3]{Is84} to be Lemma  \ref{lem:ext-ind}.
\end{enumerate}
\end{rmk}

The following is the modular version of \cite[Lemma 7.3]{FLZ20a}.

\begin{lem}\label{lem:ext-char}
	Suppose that $G=NH$ is a finite group and $N\unlhd G$, $H\le  G$ such that $M=N\cap H\unlhd N$. 
	Let $\chi\in\IBr(N)$, $\xi\in\IBr(M\mid \chi)$ and $\phi\in\IBr(N_\xi\mid\xi)$  such that $\chi=\Ind_{N_\xi}^{N}\phi$.
	Assume further that $H_\chi\le H_\xi$ and $\xi=\Res^{N_\xi}_{M}\phi$.
	Then 
	\begin{enumerate}[\rm(i)]
		\item $H_\chi=H_\phi\le H_\xi$,
		\item The following are equivalent.
		\begin{itemize}
			\item $\chi$ extends to $G_\chi$,
			\item $\xi$ extends to $H_\chi$,
			\item $\phi$ extends to $H_\chi N_\xi$.
		\end{itemize}
	\end{enumerate}
\end{lem}

\begin{proof}
This follows from a similar argument as in the proof of  \cite[Lemma 7.3]{FLZ20a} by replacing \cite[Cor.~4.2 and 4.3]{Is84}  there by Lemma \ref{lem:ext-res} and \ref{lem:ext-ind}.
\end{proof}

\subsection{Covering of weights}

We recall the Dade--Glauberman--Nagao correspondence from \cite{NS14}.
Let $N\unlhd M$ be finite groups such that $M/N$ is an $\ell$-group and $D_0$  a normal $\ell$-subgroup of $N$ contained in $Z(M)$.
Suppose that $b$ is an $M$-invariant block of $N$ with defect group $D_0$.
Let $B$ be the (unique) block of $M$ covering $b$ and $D$ be a defect group of $B$.
Let $B'$ be the Brauer correspondent of $B$.
Denote $L:=\N_N(D)$ and let $b'$ be the unique block of $L$ covered by $B'$.
Then by \cite[Thm.~5.2]{NS14}, there is a natural bijection $\pi_D:\Irr_D(b)\to \Irr_D(b')$, which is
called the \emph{Dade--Glauberman--Nagao (DGN) correspondence}.
Here, $\Irr_D(b)$ denotes the set of $D$-invariant irreducible characters in $b$.

Let $G\unlhd \tG$.
Now we recall the relationship ``covering" for weights between $\tG$ and $G$ defined in \cite{BS20}.
Let $(Q,\varphi)$ be a weight of $G$ and we set $N=\N_G(Q)$.
Let $M$ be a subgroup of $\N_{\tG}(Q)_\varphi$ such that $N\le M$ and $M/N$ is an $\ell$-group.
We fix a defect group $\tQ/Q$ of the unique block of $M/Q$ which covers $\bl_{N/Q}(\vhi)$.
Then $\tQ\cap N=Q$, $M=N\tQ$ and 
$\N_{M/Q}(\tQ/Q)=(\tQ/Q)\ti \C_{N/Q}(\tQ/Q)$.
Note that $\pi_{\tQ/Q}(\varphi)\in\dz(\C_{N/Q}(\tQ/Q))$
and denote by $\bar\pi_{\tQ/Q}(\varphi)$ the associated character in $\Irr(\N_{M/Q}(\tQ/Q)/(\tQ/Q))$ which lifts to $\pi_{\tQ/Q}(\varphi)\times 1_{\tQ/Q}\in\Irr(\N_{M/Q}(\tQ/Q))$.
Also note that  $\N_{M/R}(\tQ/Q)/(\tQ/Q)\cong \N_G(\tQ)\tQ/\tQ$ is normal in $\N_{\tG}(\tQ)/\tQ$.
Following Brough--Sp\"ath \cite{BS20}, a weight $(\tQ,\tvhi)$ of $\tG$ with $\tvhi\in\dz(\N_{\tG}(\tQ)/\tQ\mid \bar\pi_{\tQ/Q}(\varphi))$ is said to \emph{cover} $(R,\varphi)$.
In addition,
there are some Clifford-like properties for weights between $\tG$ and $G$;
see \cite[\S 2]{BS20}.

For later use, we prove the following lemma.

\begin{lem}\label{lem:cover-weight}
	Suppose that $G, \hG \unlhd \tG$ are finite groups with $G\unlhd \hG$.
	Let $(Q,\vhi)$,  $(\hQ,\hvhi)$, $(\tQ,\tvhi)$  be weights of $G$, $\hG$, $\tG$ respectively.
	If $(\tQ,\tvhi)$ covers $(\hQ,\hvhi)$ and $(\hQ,\hvhi)$ covers $(Q,\vhi)$, then $(\tQ,\tvhi)$ covers $(Q,\vhi)$.
\end{lem}

\begin{proof}
	Let $N=\N_G(Q)$, $\hM=N\hQ$ and let $B$ be the unique block of $\hM/Q$ covering $\bl_{N/Q}(\vhi)$ and $B'$ the Brauer correspondent of $B$.
	Then $\hQ/Q$ is a defect group of the block $B$ (and $B'$).
	As above, $\bl_{\N_{N}(\hQ)/Q}(\pi_{\hQ/Q}(\vhi))$ is the unique block of $\N_{N}(\hQ)/Q$ covered by $B'$ and we denote $\vhi_0:=\pi_{\hQ/Q}(\vhi)$.
	Since $$\N_{N}(\hQ)/Q\cong \N_{G}(\hQ)\hQ/\hQ\cong \N_{\hM/Q}(\hQ/Q)/(\hQ/Q),$$ one can view $\vhi_0$ as a character of $\N_{\hM/Q}(\hQ/Q)/(\hQ/Q)$, $\N_{G}(\hQ)\hQ/\hQ$ or $\N_{G}(\hQ)$, or a character in $B'$.
	Then $(\hQ,\hvhi)$ covers $(Q,\vhi)$ implies $\hvhi\in\Irr(\N_{\hG}(\hQ)\mid \vhi_0)$ (here $\vhi_0$ is viewed as a character of $\N_{G}(\hQ)\hQ$ which contains $\hQ$ in its kernel).
	
	Let $\hN=\N_{\hG}(\hQ)$, $\tM=\hN \tQ$ and let $\widetilde B$ be the unique block of $\tM/\hQ$ covering $\bl_{\hN/\hQ}(\hvhi)$ and $\widetilde B'$ the Brauer correspondent of $\widetilde B$. 
	Let $\hvhi_0=\pi_{\tQ/\hQ}(\hvhi)$. Then $\bl_{\N_{\hN}(\tQ)/\hQ}(\pi_{\tQ/\hQ}(\hvhi))$ is the unique block of $\N_{\hN}(\tQ)/\hQ$ covered by $\widetilde B'$.
	As above, we view $\hvhi_0$ as a character of $\N_{\tM/\hQ}(\tQ/\hQ)/(\tQ/\hQ)$
	and $\N_{\hG}(\tQ)$ (also, a character in $\widetilde B'$) and then $\tvhi\in\Irr(\N_{\tG}(\tQ)\mid \hvhi_0)$.
	
	Now let $\tM_0=N\tQ$ and $\widetilde B_0$ be the unique block of $\tM_0/Q$ covering $\bl_{N/Q}(\vhi)$.
	Then $\widetilde B_0$ covers $B$. Let $\widehat B_0$ be the unique block of $\N_{\tM_0/Q}(\hQ/Q)\cong \N_{\tM_0}(\hQ)/Q$ covering $B'$. 
	Then by \cite[Lemma~2.3]{KS15}, $\widehat B_0^{\tM_0/Q}=\widetilde B_0$. 
	Also, according to $\hvhi\in\Irr(\N_{\hG}(\hQ)\mid \vhi_0)$, one has  the unique block of $\N_{\tM_0}(\hQ)/\hQ$ dominated by $\widehat B_0$ is covered by $\tB$.
	
	Let $\widetilde B'_0$ be the Brauer correspondent of $\widetilde B_0$. Then 
	$(\widetilde B'_0)^{\N_{\tM_0}(\hQ)/Q}=\widehat B_0$. 
	Let $\overline{\widetilde B'_0}$, $\overline{\widehat B_0}$ be the unique blocks of $\N_{\tM_0}(\tQ)/\hQ$, $\N_{\tM_0}(\hQ)/\hQ$ dominated by $\widetilde B'_0$, $\widehat B_0$ respectively.
	Then $(\overline{\widetilde B'_0})^{\N_{\tM_0}(\hQ)/\hQ}=\overline{\widehat B_0}$ by \cite[Prop.~2.4]{Mu98} (see also \cite[Lemma~2.4]{Sp13}). 
	Now we show that $\widetilde B'$ covers $\overline{\widetilde B'_0}$.
	In fact, if $b$ is a block of $\N_{\tM_0}(\tQ)/\hQ$ covered by $\widetilde B'$, then $b^{\N_{\tM_0}(\hQ)/\hQ}$ is covered by $\tB$.
	So every block of $\N_{\tM_0}(\tQ)/\hQ$ covered by $\widetilde B'$ has the form $b^g$ with  $g\in \N_{\hG}(\tQ)/\hQ$ and all blocks $(b^g)^{\N_{\tM_0}(\hQ)/\hQ}$ are covered by $\tB$.
	So $\widetilde B'$ covers $\overline{\widetilde B'_0}$, which implies $\tvhi$ lies over $\pi_{\tQ/Q}(\vhi)$ and thus $(\tQ,\tvhi)$ covers $(Q,\vhi)$.
\end{proof}

Let $(Q,\vhi)\in\Alp^0(G)$.
Then we denote by $\Alp(\tG\mid\overline{(Q,\vhi)})$  the set of conjugacy classes of weights of $\tG$ covering $(Q,\vhi)$. 

\subsection{Modular character triples and weights}

\begin{thm}\label{thm:bij-wei-norm}
	Let $G \unlhd \tG$ with abelian quotient $\tG/G$, $(Q,\vhi)\in\Alp^0(G)$ and $(\tQ,\tvhi)\in\Alp^0(\tG)$ such that $\vhi$ extends to its stabilizer in $\N_{\tG}(Q)$ and $(\tQ,\tvhi)$ covers $(Q,\vhi)$.
	Let $M=\tQ \N_G(Q)$ and $\vhi_0=\pi_{\tQ/Q}(\vhi)\in \dz(\N_{G}(\tQ)/Q)$ be the DGN-correspondent.
	\begin{enumerate}[\rm(i)]
		\item There exists an $\N_{\tG}(Q)_\vhi$-invariant extension $\hvhi$ of $\vhi$ to $M/Q$ and a $\Lin(\N_{\tG}(Q)
		/M)$-equivariant bijection 
		$$\Delta_\vhi:\rdz(\N_{\tG}(Q)\mid \hvhi)\to \dz(\N_{\tG}(\tQ)/\tQ\mid \vhi_0)$$ such that $\bl_{\N_{\tG}(Q)}(\chi)=\bl_{\N_{\tG}(\tQ)}(\Delta_{\vhi}(\chi))^{\N_{\tG}(Q)}$ for every $\chi\in \rdz(\N_{\tG}(Q)\mid \hvhi)$, where $\vhi_0$ is regarded as a character of $\N_{G}(\tQ)\tQ/\tQ$.
		\item There exists an $\IBr(\N_{\tG}(Q)/\N_G(Q))$-equivariant bijection
		$$\Xi_\vhi:\IBr(\N_{\tG}(Q)\mid \vhi^0)\to \IBr(\N_{\tG}(\tQ)\mid (\vhi_0)^0)$$
		such that $\bl_{\N_{\tG}(Q)}(\psi)=\bl_{\N_{\tG}(\tQ)}(\Xi_{\vhi}(\psi))^{\N_{\tG}(Q)}$ for every $\chi\in \IBr(\N_{\tG}(Q)\mid \vhi^0)$ and the following diagram commutes:
				\begin{align*}
		\xymatrix{
			&\rdz(\N_{\tG}(Q)\mid \hvhi) \ar[r]^{\Delta_\vhi\ }\ar[d]_{d_{\N_{\tG}(Q)}} & \dz(\N_{\tG}(\tQ)/\tQ\mid \vhi_0)\ar[d]_{d_{\N_{\tG}(\tQ)}}\\
			&\IBr(\N_{\tG}(Q)\mid \vhi^0)\ar[r]^{\Xi_\vhi\ } & \IBr(\N_{\tG}(\tQ)\mid (\vhi_0)^0)}
		\end{align*}
		where $d_*$ sends $\chi$ to $\chi^0$ is bijective.
	\end{enumerate}
\end{thm}

\begin{proof}
	(i) is \cite[Thm.~2.10]{BS20}.  
	Note that $M/\N_G(Q)$ is a Sylow $\ell$-subgroup of $\N_{\tG}(Q)_\vhi/\N_G(Q)$.
 (ii) follows from (i) and the construction of $\Delta_\vhi$ in \cite{BS20} directly.	
\end{proof}

\begin{rmk}
Following \cite[Problem (8.10)]{Na98}, $(\N_{\tG}(Q)_\vhi, \N_G(Q),\vhi)$ and $(\N_{\tG}(\tQ)_\vhi, \N_G(\tQ),\vhi_0)$ in Theorem \ref{thm:bij-wei-norm} are ordinary-modular character triples.
In addition, these two ordinary-modular character triples  are isomorphic.
\end{rmk}

Now we prove the following technical theorem, a general statement on the modular character triples, which will be  used in
the proof of Theorem \ref{main-thm-quasi-isolated}.

\begin{thm}\label{cor:ext-wei-dgn}
In the situation of Theorem \ref{thm:bij-wei-norm}, we assume further that $A$ is a finite group such that $G,\tG\unlhd A$.
\begin{enumerate}[\rm(i)]
	\item
\begin{enumerate}[\rm(a)]
\item Then $\Delta_\vhi$ and $\Xi_\vhi$ can be chosen to be compatible with the action of the group $\N_A(Q)$ in the sense that $(\Delta_\vhi)^x$ (resp. $(\Xi_\vhi)^x$) gives a map $\Delta_{\vhi^x}$ (resp. $\Xi_{\vhi^x}$) with the properties in  Theorem \ref{thm:bij-wei-norm}.
\item 
Let $\tG\le H\le A$ be an intermediate subgroup such that  for every prime $r$ dividing $|H/\tG|$ with $r\ne\ell$ the Sylow $r$-subgroups of $H/\tG$ are abelian.
Then for any
$\psi\in\IBr(\N_{\tG}(Q)\mid \vhi^0)$,
one has that
$\psi$ extends to $\N_H(Q)_\psi$ if and only if $\Xi_\vhi(\psi)$ extends to $\N_H(\tQ)_\psi$.
\end{enumerate}
\item Suppose $B$ is a union of blocks of $G$ such that $B$ is $A$-invariant (in particular, $B$ is a union of $\tG$-orbits) and there exists a blockwise $A$-equivariant bijection $$\Omega:\IBr(B)\to\Alp(B)$$ such that $(A_\psi,G,\psi)\geqslant_b (\N_A(Q)_\vhi,\N_G(Q),\vhi^0)$, for every $\psi\in \IBr(B)$ and $\Omega(\psi)=\overline{(Q,\vhi)}$.
Assume further that for every $(Q,\vhi)\in\Alp^0(B)$, $\vhi$ extends to its stabilizer in $\N_{\tG}(Q)$.
Let $\tcB$ be the union of blocks of $\tG$ covering the blocks in $B$.
Then there exists a blockwise $(\IBr(\tG/G)\rtimes A)$-equivariant bijection $$\wOm:\IBr(\tcB)\to\Alp(\tcB).$$

Let $\tG\le H\le A$ be an intermediate subgroup such that  for every prime $r$ dividing $|H/\tG|$ with $r\ne\ell$ the Sylow $r$-subgroups of $H/\tG$ are abelian.
Then for every $\tpsi\in \IBr(\tcB)$ and $\wOm(\tpsi)=\overline{(\tQ,\tvhi)}$, 
$\tpsi$ extends to  $H_{\tpsi}$  if and only if $\tvhi$ extends to $\N_H(\tQ)_{\tvhi}$.
\end{enumerate}
\end{thm}

\begin{proof}
(i)  We first consider (a) and  it suffices to prove it for the $\N_A(Q)$-conjugacy class of a fixed character $\vhi$.
 First by \cite[Thm.~5.13]{NS14}, $\Delta_\vhi$ is $\N_A(Q)_\vhi$-equivariant.
 Then for every $x\in \N_A(Q)$, we consider $\vhi'=\vhi^x$.
Suppose that $y\in \N_A(Q)$ is another element such that $\vhi'=\vhi^y$.
Then $xy^{-1}\in \N_A(Q)_\vhi$.
So $\Delta_\vhi(\chi^{x^{-1}})^{xy^{-1}}=\Delta_\vhi(\chi^{y^{-1}})$ and hence $\Delta_\vhi(\chi^{x^{-1}})^{x}=\Delta_\vhi(\chi^{y^{-1}})^y$.
Thus we have shown that $\Delta_\vhi$ is well-defined by $\Delta_{\vhi^x}(\chi)=\Delta_{\vhi}(\chi^{x^{-1}})^x$.
It is straightforward to check that (a) holds.
For (b), thanks to  Lemma \ref{lem:ext-def-zero} and
Theorem \ref{thm:bij-wei-norm}, we transfer to ordinary characters: let $\chi\in\rdz(\N_{\tG}(Q)\mid\hvhi)$ with $\psi=\chi^0$. 
Then it is equivalent to  show that 
$\chi$ extends to $\N_H(Q)_\chi$ if and only if $\Delta_\vhi(\chi)$ extends to $\N_H(\tQ)_\chi$.
We note that $\N_H(Q)_{\chi}\subseteq \N_{\tG}(Q)\N_H(Q)_\vhi=:L$.
Then $\N_{\tG}(Q)_\vhi\unlhd L$.
Let $\eta\in\Irr(\N_{\tG}(Q)_\vhi)$, $\eta_0\in\Irr(\N_{\tG}(\tQ)_\vhi)$ such that $\Ind^{\N_{\tG}(Q)}_{\N_{\tG}(Q)_\vhi}\eta=\chi$ and 
$\Ind^{\N_{\tG}(\tQ)}_{\N_{\tG}(\tQ)_\vhi}\eta_0=\Delta_\vhi(\chi)$.
According to Lemma \ref{lem:ext-indu}, it suffices to show that $\eta$ extends to $L_\eta$ if and only if $\eta_0$ extends to $\N_L(\tQ)_{\eta_0}$.
Since $\eta$ and $\eta'$ are extensions of $\vhi$ and $\vhi_0$ respectively, we only need to show that  $\vhi$ extends to $L_\eta$ if and only if $\vhi_0$ extends to $\N_L(\tQ)_{\eta_0}$.
This follows by \cite[Thm.~5.13]{NS14}.

Now we consider (ii).
For every $\psi\in\IBr(B)$ and $\Omega(\psi)=\overline{(Q,\vhi)}$, we denote by $\sigma^{(\psi)}$ an isomorphism of the modular character triples  $(A_\psi,G,\psi)$ and $(\N_A(Q)_\vhi,\N_G(Q),\vhi^0)$ defined by the block isomorphism
$$(A_\psi,G,\psi)\geqslant_b (\N_A(Q)_\vhi,\N_G(Q),\vhi^0).$$ 
By an analogous argument as in the proof of \cite[Prop.~4.7(a)]{NS14}, we can choose  $\sigma^{(\psi)}$ such that
\begin{equation}\label{equ:sigma-equiva}
\sigma^{(\psi)}_J(\zeta)^x=\sigma^{(\psi^x)}_{J^x}(\zeta^x)\ \textrm{for every}\ G\le J\le A_\psi, \psi\in\IBr(B), \zeta\in\IBr(J\mid\psi), x\in \N_A(Q).
\addtocounter{thm}{1}\tag{\thethm}
\end{equation}
By Lemma \ref{lem:char-triple-iso-exten}, the bijection
 $$\sigma_{\tG_\psi}^{(\psi)}:\IBr(\tG_\psi\mid\psi)\to\IBr(\N_{\tG}(Q)_\vhi\mid \vhi^0)$$  satisfies for every $\widehat \psi\in\IBr(\tG_\psi\mid\psi)$ and $\widehat{\vhi^0}=\sigma_{\tG_\psi}^{(\psi)}(\widehat \psi)$ that
 $$((A_\psi)_{\widehat \psi},\tG_\psi,\widehat \psi)\geqslant_b ((\N_A(Q)_\vhi)_{\widehat{\vhi^0}},\N_{\tG}(Q)_\vhi,\widehat{\vhi^0})$$
and $\bl(\widehat \psi)=(\bl(\widehat{\vhi^0}))^{\tG_\psi}$.

For $\tpsi\in\IBr(\tcB)$, let $\psi\in\IBr(G\mid \tpsi)$ and $\widehat \psi\in\IBr(\tG_{\psi}\mid\psi)$ such that $\Ind^{\tG}_{\tG_\psi}\widehat \psi=\tpsi$.
Let $\psi'\in\IBr(G\mid \tpsi)$ and $\widehat \psi'\in\IBr(\tG_{\psi'}\mid\psi')$ be another pair of Brauer characters such that $\Ind^{\tG}_{\tG_{\psi'}}\widehat \psi'=\tpsi$.
Then there exists $g\in \tG$ such that $\psi'=\psi^g$ and $\widehat \psi'=\widehat \psi^g$.
Since $\Omega$ is $\tG$-equivariant, we may assume that $g\in \N_{\tG}(Q)$.
According to (\ref{equ:sigma-equiva}) one has
$\sigma^{(\psi')}_{\tG_{\psi'}}(\widehat\psi')=\sigma^{(\psi^g)}_{\tG_{\psi^g}}(\widehat\psi^g)=\sigma^{(\psi)}_{\tG_{\psi}}(\widehat\psi)^g$.
Then the induction to $\N_{\tG}(Q)$ of $\sigma^{(\psi')}_{\tG_{\psi'}}(\widehat\psi')$ and $\sigma^{(\psi)}_{\tG_{\psi}}(\widehat\psi)$ coincide.
So the Brauer character $\tvhi:=\Ind^{\N_{\tG}(Q)}_{\N_{\tG}(Q)_\vhi}\sigma^{(\psi)}_{\tG_{\psi}}(\widehat\psi)$ is well-defined and independent of the choice of $\psi$.
In addition, $\bl(\tvhi)^{\tG}=\bl(\tpsi)$.

Let $\mathcal T$ be a fixed complete set of representatives of $\tG$-orbits in $\IBr(B)$.
Then by the above paragraph, for every $\psi\in\mathcal T$ and $\Omega(\psi)=\overline{(Q,\vhi)}$,
 we can define a map $$\Pi_{\psi}:\ \IBr(\tG\mid \psi)\to \IBr(\N_{\tG}(Q)\mid \vhi^0),\quad
 \tpsi\mapsto \tvhi.$$
 In addition, $\N_{A}(Q)_{\tpsi}=\N_{A}(Q)_{\Pi_\psi(\tpsi)}$ and $\bl(\widetilde \psi)=(\bl(\widetilde{\vhi}))^{\tG}$.
By Proposition \ref{prop:indu-triple},
	$$(A_{\tpsi},\tG,\tpsi)\geqslant_b (\N_{A}(Q)_{\tvhi},\N_{\tG}(Q),\tvhi).$$
In particular, for  an intermediate subgroup $\tG\le H\le A$ such that  for every prime $r$ dividing $|H/\tG|$ with $r\ne\ell$ the Sylow $r$-subgroups of $H/\tG$ are abelian,
$\tpsi$ extends to  $H_{\tpsi}$  if and only if $\Pi_\psi(\tpsi)$ extends to $\N_H(\tQ)_{\tpsi}$.
By Clifford theory of Brauer characters and the Clifford-like properties of weights (cf. \cite[\S 2]{BS20}), we see that $\Pi_\psi$ is bijective.
Together with $\Xi_\vhi$, we have obtained a bijection between $\IBr(\tG\mid \psi)$ and $\Alp(\tG\mid \Omega(\psi))$.

Therefore, when $\psi$ runs through $\mathcal T$, one has a bijection $\wOm:\IBr(\tcB)\to\Alp(\tcB)$ and it can be checked directly that $\wOm$ is $(\IBr(\tG/G)\rtimes A)$-equivariant and satisfies the properties as described.
This completes the proof.
\end{proof}

\subsection{A criterion for the inductive BAW condition}

In \cite{Sp17} the inductive conditions for some of the local-global conjectures
were rephrased in terms of character triples.
Let $S$ be a finite non-abelian simple group, $\ell$ a prime dividing $|S|$,
 $G$  the universal $\ell'$-covering group of $S$, and $B$ an $\ell$-block of $G$.
By  \cite[Thm.~4.4]{Sp17}, $B$ is \emph{BAW-good} if for $\Ga=\Aut(G)_B$,
\begin{itemize}
\item there exists a $\Ga$-equivariant bijection $\Omega_B:\IBr(B)\to\Alp(B)$, and 
\item for every $\psi\in\IBr(G)$ and $\overline{(Q,\vhi)}=\Omega_B(\psi)$, one has 
\begin{equation}	\label{iso:quotient}
(G/Z\rtimes \Ga_\psi,G/Z,\overline \psi)\geqslant_b( (G/Z\rtimes \Ga)_{Q,\vhi}, \N_G(Q)/Z,\overline{\vhi^0} ),
\addtocounter{thm}{1}\tag{\thethm}
\end{equation}
 where $Z=\ker(\psi)\cap Z(G)$, and $\overline\psi$ and $\overline{\vhi^0}$ lift to $\psi$ and $\vhi^0$ respectively.
\end{itemize}
We say the simple group $S$ is \emph{BAW-good} (for the prime $\ell$) if every $\ell$-block of $G$ is BAW-good.

We note that by a similar result as \cite[Lemma~3.2]{NS14} for Brauer characters,
(\ref{iso:quotient}) implies Definition \ref{def:iBAW-bijection} (ii) but the converse is not necessarily true.
Thus if the block $B$ is BAW-good, then there exists an iBAW-bijection for $B$.

We end this section by recalling the following criterion
given by Brough and Sp\"ath \cite{BS20}  of the inductive BAW condition, which is useful for simole groups of Lie type, especially when the outer automorphism groups are non-cyclic.
For a finite group $G$, we dentoe by $\Bl(G)$ the set consisting of blocks of $G$.

\begin{thm}[Brough--Sp\"ath]\label{thm:criterion-block}
	Let $S$ be a finite non-abelian simple group and $\ell$ a prime dividing $|S|$.
	Let $G$ be the universal covering group of $S$, $B\subseteq \Bl(G)$ be a subset and assume	there are groups $\tG$, $E$ such that 
	 $G \unlhd \tG  E$, $B$ is $\tG$-invariant,
	and the following hold.
	\begin{enumerate}[\rm(i)]\setlength{\itemsep}{0pt}
		\item \begin{enumerate}[\rm(a)]\setlength{\itemsep}{0pt}
			\item $G=[\tG,\tG]$,
			$E$ stabilizes  $\tG$
			 and $E$ is abelian or isomorphic to the direct product of a cyclic
			group and the symmetric group $\mathfrak S_3$ on 3 elements,
			\item $\C_{\tG E}(G)=Z(\tG)$ and $\tG E/Z(\tG) \cong \Aut(G)$,
			\item any element of $\IBr(B)$ extends to its stabilizer in $\tG$, and
			\item for any $(Q,\varphi)\in \Alp^0(B)$, the weight character $\varphi$ extends to its stabilizer in $\N_{\tG}(Q)$.
			\end{enumerate}
	\item 	There exists a blockwise $\tG\rtimes E_B$-equivariant bijection
	$\Omega: \IBr(B) \to \Alp(B)$.
		\item Let $\mathcal{\widetilde B}$ be the union of $\ell$-blocks of $\tG$ covering the blocks in $B$.
		There exists an $\Lin_{\ell'}(\tG/G) \rtimes E_B$-equivariant bijection
		$\wOm: \IBr(\mathcal{\widetilde B}) \to \Alp(\mathcal{\widetilde B})$ such that
		\begin{enumerate}[\rm(a)]\setlength{\itemsep}{0pt}
				\item $\wOm(\IBr(\tG\mid \psi))=\Alp(\tG\mid \Omega(\psi))$ for every $\psi\in \IBr(B)$, 
				\item $\wOm(\IBr(\tG\mid\nu^0)\cap \IBr(\mathcal{\tB})) = \Alp(\tG\mid\nu)\cap \Alp(\mathcal{\tB})$ for every $\nu \in \Lin_{\ell'}(Z(\tG))$, and
			\item for every $\psi \in\IBr(\cB)$, there exists $\tpsi\in \IBr(\tG\mid\psi)$ such that  $\bl(\hpsi)=\bl(\hvhi)^{\tG_\psi}$ where $\hpsi\in\IBr(\tG_\psi\mid\psi)$ is the  Clifford correspondent of $\tpsi$,
			$\Omega(\psi)=\overline{(Q,\vhi)}$, $\wOm(\tpsi)=\overline{(\tQ,\tvhi)}$,
     $\hvhi$ is an extension of $\vhi$ to $\N_{\tG}(Q)_\vhi/Q$ satisfying that via induction and $\Delta_\vhi$ it corresponds to $(\tQ,\tvhi)$.
		\end{enumerate}

		\item In every $\tG$-orbit of Brauer characters in $\IBr(B)$, there exists some $\psi_0$ such that
		\begin{enumerate}[\rm(a)]\setlength{\itemsep}{0pt}
			\item $(\tG\rtimes E)_{\psi_0}=\tG_{\psi_0}\rtimes E_{\psi_0}$ and
		$\psi_0$ extends to $G\rtimes E_{\psi_0}$, and

		\item for $\overline{(Q,\varphi_0)}=\Omega(\psi_0)$, one has
	 $(\tG E)_{Q,\varphi_0} = \tG_{Q,\varphi_0} (GE)_{Q,\varphi_0}$ and
		 $(\varphi_0)^0$ extends to $(G\rtimes E)_{Q,\varphi_0}$.
		\end{enumerate}
	\end{enumerate}
	Then all blocks $b\in B$ are BAW-good.
\end{thm}

\begin{proof}
This is in fact proved in \cite[\S 4.B]{BS20}, see the proof of \cite[Thm.~4.5 and Lemma~4.6]{BS20}. See also \cite{FLZ20b}.
\end{proof}


\section{Finite groups of Lie type and equivariant Morita equivalences}\label{sec:lie-type-equivalence}

In this section, we first recall some basic facts in the representation theory of finite groups of Lie type and then establish a Jordan decomposition for weights.
The local subgroups of a connected reductive group may be  disconnected.
So we start from algebraic groups which are possibly disconnected.
Fix a prime $p$ and an algebraic closure $\overline{\F}_p$ of the field $\F_p$.
We always assume that $p\ne\ell$.

\subsection{Notation}

Let $\bX$ be a quasi-projective  algebraic variety over $\overline{\F}_p$ and let $G$ be a finite group acting on $\bX$.
Then there is an object $\G\Ga_c(\bX,\Lambda)$ of $\mathrm{Ho}^b(\Lambda G\textrm{-perm})$ which is unique up to isomorphism as in \cite[2.C]{BDR17}. 
Denote by $\R\Ga_c(\bX,\Lambda)$ the image of  $\G\Ga_c(\bX,\Lambda)$ in $D^b(\Lambda G)$ (the global section complex).
Moreover, we denote by $H^d_c(\bX,\Lambda)\in \Lambda G$-mod the $d$-th cohomology module of the complex $\R\Ga_c(\bX,\Lambda)$.
Suppose that $\bX$ is of dimension $n$.
We write 
$H^{\dim}_c(\bX,\Lambda)=H^n_c(\bX,\Lambda)$.

By \cite[Thm.~4.2]{Ri94}, if $Q$ is an $\ell$-subgroup of $G$,
then the inclusion $\bX^Q\hookrightarrow \bX$ induces an isomorphism
$$\G\Ga_c(\bX^Q,k)\to \Br_Q(\G\Ga_c(\bX,\Lambda))$$
in $\mathrm{Ho}^b(k\N_G(Q)\textrm{-perm})$.

\subsection{Finite groups of Lie type and representations}

Let $\bG$ be a (possibly disconnected) reductive algebraic group defined over $\overline{\F}_p$. We refer to \cite{DM94} and \cite{BDR17} for some basic results on disconnected algebraic groups.
Recall that $\Lambda\in\{\cO,k\}$.

Denote by $\bG^\circ$ the connected component of $\bG$ containing the identity.
A torus of $\bG$ is a torus of $\bG^\circ$.
A closed subgroup $\bP$ of $\bG$ is called a \emph{parabolic subgroup} if the variety $\bG/\bP$ is complete.
If $\bP$ is a parabolic subgroup of $\bG$ , then $\bP^\circ=\bP\cap\bG^\circ$.
In addition, a closed subgroup of $\bP$ is a parabolic subgroup of $\bG$ if and only if $\bP^\circ$ is a parabolic subgroup of $\bG^\circ$.

Suppose that $\bP$ is a parabolic subgroup of $\bG$.
Let $\bU$ be its unique maximal connected unipotent normal subgroup and we call $\bU$ the \emph{unipotent radical} of $\bP$.
The \emph{Levi complement} of $\bP$ with respect to $\bU$ is a subgroup $\bL$ such that $\bP=\bU\rtimes \bL$.
We note that $\bU$ is also the unipotent radical of $\bP^\circ$.
Then $\bP^\circ=\bU\rtimes \bL^\circ$ is a Levi decomposition and $\bL=\N_{\bP}(\bL^\circ)$.

Let $Q$ be a finite solvable $p'$-subgroup of $\bG$.
Then by \cite[Remark 3.5]{BDR17}, $\N_{\bG}(Q)$ is a reductive group.
Let $\bP$ be a parabolic subgroup of $\bG$ with Levi decomposition $\bP=\bU\rtimes \bL$.
Assume further that $Q$ normalizes $\bP$ and $\bL$. 
Then $\N_{\bP}(Q)$ is a parabolic subgroup of $\N_{\bG}(Q)$ with unipotent radical $\C_{\bU}(Q)$ and Levi complement $\N_{\bL}(Q)$, \emph{i.e.}, $\N_{\bP}(Q)=\C_{\bU}(Q)\rtimes \N_{\bL}(Q)$.
Similarly, by \cite[Prop.~3.4]{BDR17},
$\C_{\bG}(Q)$ is a reductive group and $\C_{\bP}(Q)$ is a parabolic subgroup of $\C_{\bG}(Q)$ with Levi decomposition $\C_{\bG}(Q)=\C_{\bU}(Q)\rtimes \C_{\bL}(Q)$.

Let $F:\bG\to\bG$ be a Steinberg map, a power of which is a Frobenius endomorphism defining an $\F_q$-structure on $\bG$, where $q$ is a power of $p$.
Let $\bP$ be a parabolic subgroup of $\bG$ with unipotent radical $\bU$ and 
$F$-stable Levi complement $\bL$.
From now on we assume that $\ell$ is a prime different from $p$.
The \emph{Deligne--Lusztig varieties} were introduced in \cite{DL76} and we recall the definition as follows
$$\bY_{\bU}^{\bG}=\{ g\bU\in \bG/\bU\mid g^{-1}F(g)\in \bU\cdot F(\bU)  \}.$$
Note that $\bY_{\bU}^{\bG}$ has a left  $\bG^F$- and right $\bL^F$-action.

Let $\sigma:\bG\to\bG$ be a bijection morphism of algebraic groups with commutes with the action of $F$, \emph{i.e.}, $\sigma\circ F=F\circ \sigma$.
Let $\bP$ be a parabolic subgroup of $\bG$ with Levi decomposition $\bP=\bU\rtimes \bL$ such that $F(\bL)=\bL$.
Then we have:

\begin{lem}\label{lem:action-lie}
$\sigma$ induces an isomorphism $$\sigma^*:\G\Gamma_c(\bY_{\sigma(\bU)}^{\bG},\Lambda)\to {}^\sigma \G\Gamma_c(\bY_{\bU}^{\bG},\Lambda)^\sigma$$ in $\mathrm{Ho}^b(\Lambda[\bG^F\ti(\bL^F)^{\opp}])$.
\end{lem}

\begin{proof}
Note that $\sigma(\bP)$ is a 	parabolic subgroup of $\bG$ with $F$-stable Levi subgroup $\sigma(\bL)$ and unipotent radical $\sigma(\bU)$.
This follows by a theorem of Rouquier \cite[Thm.~2.12]{Ro02}; see \cite[Lemma~2.23]{Ru20b} for details.
\end{proof}

\subsection{Splendid Rickard equivalences and Brauer categories}

From now on we assume that $\bG$ is a connected reductive group.
Let $(\bG^*,F^*)$ be Langlands dual to $(G,F)$ and we also dentoe $F^*$ by $F$ in the following.

Let $s$ be a semisimple $\ell'$-element of ${\bG^*}^F$.
Then by a theorem of Brou\'e--Michel \cite{BM89}, there is a central idempotent $e_{s}^{\bG^F}$ of $\Lambda \bG^F$, which is uniquely determined by the ${\bG^*}^F$-conjugacy class containing $s$; 
see also \cite[\S9.2]{CE04}.
If $b$ is a block idempotent of $Z(\Lambda \bG^Fe_s^{\bG^F})$, then we say the block \emph{$b$ is in $s$}. 

Now assume that $\bL^*$ is an $F$-stable Levi subgroup  of $\bG^*$ containing $\C_{{\bG^*}^F}(s)\C^\circ_{\bG^*}(s)$ throughout this section. 
Let $\bL$ be the $F$-stable Levi subgroup of $\bG$ dual to $\bL^*$.
It was proven by Bonnaf\'e--Rouquier \cite{BR03} that the bimodule $H^{\dim}_c(\bY_{\bU}^{\bG},\cO)e_{s}^{\bL^F}$ induces a Morita equivalence between $\cO \bL^F e_{s}^{\bL^F}$ and $\cO \bG^F e_{s}^{\bG^F}$.
This was improved by Bonnaf\'e, Dat and Rouquier and we state their result as follows: 

\begin{thm}[Bonnaf\'e--Dat--Rouquier, {\cite[Thm.~7.7]{BDR17}}]\label{thm:equ-blocks}
The complex $\mathcal C:=\G\Ga_c(\bY^{\bG}_{\bU},\cO)^{\rm{red}}e_s^{\bL^F}$ of $\cO \bG^F e_s^{\bG^F}$-$\cO \bL^F e_s^{\bL^F}$-bimodules induces a splendid Rickard equivalence between $\cO \bL^F e_s^{\bL^F}$ and $\cO \bG^F e_s^{\bG^F}$.
The bimodule $H^{\dim(\bY^{\bG}_{\bU})}(\mathcal C)$ induces a Morita equivalence between $\cO \bL^F e_s^{\bL^F}$ and $\cO \bG^F e_s^{\bG^F}$.
 \end{thm}
 
 Let $b$ be a block idempotent of $Z(\cO \bG^F e_s^{\bG^F})$ and $c\in Z(\cO \bL e_s^{\bL^F})$ the block idempotent corresponding to $b$ under the Morita equivalence between $\cO \bL e_s^{\bL^F}$ and $\cO \bG^F e_s^{\bG^F}$ given in Theorem \ref{thm:equ-blocks}.
 By a theorem of Puig \cite[Thm.~19.7]{Pu99}, the Brauer categories of splendid Rickard equivalent blocks are equivalent. 
 Let $(Q,c_Q)$ be a $c$-Brauer pair.
 By \cite[Thm.~1.7]{Ha99}~(see also \cite[\S 1.6]{Ru20a}),  there is a unique $b$-Brauer pair $(Q,b_Q)$ such that the complex $b_Q \Br_{\Delta Q}(\mathcal C)c_Q$ induces a Rickard equivalence between $k \C_{\bL^F}(Q)c_Q$ and  $k\C_{\bG^F}(Q)b_Q$.
 For any other $b$-Brauer pair $(Q,b_Q')$ one has $b_Q' \Br_{\Delta Q}(\mathcal C)c_Q=0$ in $\mathrm{Ho}^b(k[\C_{\bG^F}(Q)\times \C_{\bL^F}(Q)^{\opp}])$.
 For any other $c$-Brauer pair $(Q,c_Q')$ one has $b_Q \Br_{\Delta Q}(\mathcal C)c_Q'=0$ in $\mathrm{Ho}^b(k[\C_{\bG^F}(Q)\times \C_{\bL^F}(Q)^{\opp}])$.
 
 We also view the Brauer categories $\bBr(\bL^F,c)$ and $\bBr(\bG^F,b)$ as sets of their objects.
 Then we have defined a map as follows:
 
 \begin{lem}\label{lem:inj-Bra-cat}
The map $\Theta$ given by $(Q,c_Q)\mapsto (Q,b_Q)$ from $\bBr(\bL^F,c)$ to $\bBr(\bG^F,b)$ is injective.
 \end{lem}

 Note that for any $\ell$-subgroup $Q$ of $\bL^F$, if the complex $\Br_{\Delta Q}(\mathcal C)c_Q\simeq \G\Ga_c(\bY_{\C_{\bU}(Q)}^{\C_{\bG}(Q)},k)$	induces a Rickard equivalence between the block algebras $k\C_{\bG^F}(Q)b_Q$ and $k\C_{\bL^F}(Q)c_Q$, then by \cite[Prop.~3.9]{Ru20a}, $H^{\dim}_c(\bY_{\C_{\bU}(Q)}^{\C_{\bG}(Q)},k)c_Q$ induces a Morita equivalence between  $k\C_{\bG^F}(Q)b_Q$ and $k\C_{\bL^F}(Q)c_Q$.
From this, we can define $\Theta$ by setting $(Q,b_Q)$ to be the unique $b$-Brauer pair of $\bG^F$ such that $H^{\dim}_c(\bY_{\C_{\bU}(Q)}^{\C_{\bG}(Q)},k)c_Q$ induces a Morita equivalence between  $k\C_{\bG^F}(Q)b_Q$ and $k\C_{\bL^F}(Q)c_Q$.

 \begin{rmk}\label{remark:Brauer-cat-splendid}
 We fix a defect group $D$ of the block $c$ and a maximal $c$-Brauer pair $(D,c_D)$, then $D$ is a defect group of the block $b$.
 	In addition, there exists a $b$-Brauer pair $(D,b_D)$ such that if $(Q,c_Q)\le (D,c_D)$ is a $c$-Brauer pair then the $b$-Brauer pair  $(Q,b_Q)\le (D,b_D)$ is just the (unique) $b$-Brauer pair given as in the above paragraph. 	
Then the restriction of
the map $\Theta: (Q,c_Q)\mapsto (Q,b_Q)$ to $\cF_{(D,c_D)}(\bL^F,c)$ induces an isomorphism between the fusion systems $\cF_{(D,c_D)}(\bL^F,c)$ and $\cF_{(D,b_D)}(\bG^F,b)$. 
Also, one has  (see \cite[Cor.~1.6]{Ru20a}) $$\N_{\bL^F}(Q,c_Q)/\C_{\bL^F}(Q)\cong \N_{\bG^F}(Q,b_Q)/\C_{\bG^F}(Q).$$

 We recall the equivalence relations (which are denoted by ``$\sim$") on the sets $\cF_{(D,c_D)}(\bL^F,c)$ and $\cF_{(D,b_D)}(\bG^F,b)$ as in Remark \ref{remark:wei-fusion-system}. 
If $\mathcal R$ is a set of subgroups of $D$ such that
$\{ (Q,c_Q)\mid  Q\in \mathcal R,\ (Q,c_Q)\le (D,c_D) \}$ is a complete set of representatives of the equivalence classes of $\cF_{(D,c_D)}(\bL^F,c)/\!\!\sim$, then 
$\{ (Q,b_Q)\mid Q\in \mathcal R,\ (Q,b_Q)\le (D,b_D) \}$ is a complete set of representatives of the equivalence classes of $\cF_{(D,b_D)}(\bG^F,b)/\!\!\sim$.
 \end{rmk}

 We denote by 
 $\IBr(\bG^F,e_{s}^{\bG})$ the union of the  $\IBr(b)$, where $b$ runs through the blocks of $\bG^F$  in $s$.
 We denote by $R_{\bL}^{\bG}$ the bijection from $\IBr(\bL^F,e_{s}^{\bL})$ to $\IBr(\bG^F,e_{s}^{\bG})$ induced by the bimodule $H^{\dim(\bY^{\bG}_{\bU})}(C)$ as in Theorem \ref{thm:equ-blocks}.

 Let $\tbG$ be another connected reductive group defined over $\overline{\F}_p$ , with Steinberg map denoted by $\widetilde F:\tbG\to\tbG$.
 Assume that there is a homomorphism $\iota:\bG\to\tbG$ of algebraic groups such that $\iota\circ F=\widetilde F\circ \iota$.
Recall that $\iota:\bG\to\tbG$ is a \emph{regular embedding} if $\iota(\bG)$ is a closed subgroup of $\tbG$,
 $\iota$ is an isomorphism of $\bG$ with $\iota(\bG)$, $Z(\tbG)$ is connected, and $\tbG$ and $\iota(\bG)$ have the same derived subgroup.
  In this way, we also denote $\widetilde F$ by $F$.
  
A regular embedding $\iota:\bG\hookrightarrow\tbG$ induces a surjection $\iota^*:\tbG^*\twoheadrightarrow \bG^*$ (see for example \cite[\S1.7]{GM20}). 
Let $s$ be a semisimple $\ell'$-element of ${\bG^*}^F$ and $b$ a block of $\bG^F$ in $s$.
If $\tilde b$ is a  block of $\tbG^F$ covering $b$, then by \cite[Prop.~2.6.16]{GM20} there exists a semisimple $\ell'$-element $\ts$ of $(\tbG^*)^F$ such that $\iota^*(\ts)=s$ and $\tilde b$ is in $\ts$.

\begin{lem}\label{lem:equ-ind-mul}
	Suppose that $\iota:\bG\hookrightarrow\tbG$ is a regular embedding.
	Let $\ts$ be a semisimple $\ell'$-element of $(\tbG^*)^F$ and $s=\iota^*(\ts)$.
	\begin{enumerate}[\rm(i)]
		\item $\Ind^{\tbG^F}_{\bG^F}\circ R^{\bG}_{\bL}(\psi)=R^{\tbG}_{\tbL}\circ \Ind^{\tbL^F}_{\bL^F}(\psi)$ for every $\psi\in\IBr(\bL^F,e_s^{\bL^F})$.
		\item $\la R^{\tbG}_{\tbL}(\tpsi)=R^{\tbG}_{\tbL}(\Res^{\tbG^F}_{\tbL^F}(\la)\tpsi)$ for every $\tpsi\in\IBr(\tbL^F,e_{\ts}^{\tbL^F})$ and $\la\in\IBr(\tbG^F/\bG^F)$.
	\end{enumerate}
\end{lem}

\begin{proof}
	Similar results for ordinary characters are classical (see, \emph{e.g.}, Corollary 3.3.25 and Proposition 3.3.26  of \cite{GM20}).
	Then the assertion follows directly since 
for every block $b$ in $s$	the set $\{\chi^0\mid\chi\in\Irr(b)\}$ spans the space
$\mathbb Q\IBr(b)$ and all maps are linear.		
\end{proof}	

We recall that an element $t$ in $\bG^*$ is called \emph{quasi-isolated} if the centralizer $\C_{\bG^*}(t)$ is not contained in a proper Levi subgroup of $\bG^*$.
Following Definition 3.1 of \cite{Ru20b}, we say that an element $t\in{\bG^*}^F$ is \emph{strictly quasi-isolated} if $\C_{{\bG^*}^F}(t)\C^\circ_{\bG^*}(t)$ is not contained in a proper Levi subgroup of $\bG^*$.
An $\ell$-block of $\bG^F$ is called a \emph{quasi-isolated block} (resp. \emph{strictly quasi-isolated}) if it is in some quasi-isolated (resp. strictly quasi-isolated) semisimple $\ell'$-element.

\subsection{Considering the automorphisms}
Let $s\in{\bG^*}^F$ be a semisimple $\ell'$-element such that $\C_{{\bG^*}^F}(s)\C^\circ_{\bG^*}(s)$   is contained in some $F$-stable Levi subgroup $\bL^*$ of $\bG^*$.
Suppose that $F_0:\tbG\to\tbG$ is a Frobenius endomorphism and $\sigma:\tbG\to\tbG$  is a bijective morphism
such that $F$ is a power of $F_0$ and  $F_0\circ \sigma=\sigma\circ F_0$ as morphisms of $\tbG$.
Let $\cA:=\langle F_0,\sigma\rangle$ and we assume further that $\cA$ stabilizes $e_s^{\bG^F}$ and there is a Levi subgroup $\bL$ of $\bG$ in duality with $\bL^*$ such that $\cA$ stabilizes $\bL$ and $e_s^{\bL^F}$.

 \begin{lem}\label{lem:equivariant-Br-pairs}
 	In this situation,
	the map $\Theta$   from Lemma \ref{lem:inj-Bra-cat} is $(\tbL^F\cA)_c$-equivariant.
\end{lem}

\begin{proof}
	Let $(Q,c_Q)$ be a $c$-Brauer pair of $\bL^F$ and $(Q,b_Q)=\Theta(Q,c_Q)$.
Then	$H^{\dim}_c(\bY_{\C_{\bU}(Q)}^{\C_{\bG}(Q)},k)c_Q$ induces a Morita equivalence between the block algebras  $k\C_{\bG^F}(Q)b_Q$ and $k\C_{\bL^F}(Q)c_Q$.
	
Let $\tau \in(\tbL^F\cA)_c$. 
Then
${}^\tau H^{\dim}_c(\bY_{\C_{\bU}(Q)}^{\C_{\bG}(Q)},k)^\tau\tau(c_Q)$ induces a Morita equivalence between the block algebras $k\C_{\bG^F}(\tau(Q))\tau(b_Q)$ and $k\C_{\bL^F}(\tau(Q))\tau(c_Q)$.
By Lemma \ref{lem:action-lie} one has that 
$\tau$ induces an isomorphism $$\tau^*:\G\Gamma_c(\bY_{\tau(\bU)}^{\bG},\Lambda)\to {}^\tau \G\Gamma_c(\bY_{\bU}^{\bG},\Lambda)^\tau$$ in $\mathrm{Ho}^b(\Lambda[\bG^F\ti(\bL^F)^{\opp}])$.
Applying the Brauer functor $\Br_{\Delta \tau(Q)}$ and taking the cohomology modules, one has $$H^{\dim}_c(\bY_{\C_{\tau(\bU)}(\tau(Q))}^{\C_{\bG}(\tau(Q))},k)\tau(c_Q)\simeq {}^\tau H^{\dim}_c(\bY_{\C_{\bU}(Q)}^{\C_{\bG}(Q)},k)^\tau\tau(c_Q).$$ 
On the other hand, by  \cite[Thm.~2.27]{Ru20a}, $$H^{\dim}_c(\bY_{\C_{\tau(\bU)}(\tau(Q))}^{\C_{\bG}(\tau(Q))},k)\tau(c_Q)\simeq H^{\dim}_c(\bY_{\C_{\bU}(\tau(Q))}^{\C_{\bG}(\tau(Q))},k)\tau(c_Q).$$
Therefore, $H^{\dim}_c(\bY_{\C_{\bU}(\tau(Q))}^{\C_{\bG}(\tau(Q))},k)\tau(c_Q)$ 
induces a Morita equivalence between the block algebras $k\C_{\bG^F}(\tau(Q))\tau(b_Q)$ and $k\C_{\bL^F}(\tau(Q))\tau(c_Q)$.
Thus according to the definition of $\Theta$, one has $$\Theta(\tau(Q),\tau(c_Q))=(\tau(Q),\tau(b_Q)).$$
This completes the proof.
\end{proof}

\begin{rmk}\label{remark:conj-Br-pairs-blocks}
	By Lemma \ref{lem:equivariant-Br-pairs},  $\Theta$ induces a bijection between the $\bL^F$-conjugacy classes of $\bBr(\bL^F,c)$ and  the $\bL^F$-conjugacy classes of $\Theta(\bBr(\bL^F,c))$. 
	By Remark \ref{remark:Brauer-cat-splendid}	,
	the number of  $\bL^F$-conjugacy classes of $\bBr(\bL^F,c)$ is equal to the number of $\bG^F$-conjugacy classes of $\bBr(\bG^F,b)$. 
	Therefore, if  $(Q,b_Q),(Q',b_{Q'})\in \Theta(\bBr(\bL^F,c))$ are two $b$-Brauer pairs,
	then $(Q,b_Q)$ and $(Q',b_{Q'})$ are $\bG^F$-conjugate if and only if they are $\bL^F$-conjugate.
\end{rmk}

For a block $C_Q$ of $\N_{\bL^F}(Q)$ with $(C_Q)^{\bL^F}=c$, 
we let $c_Q$ be a block of $\C_{\bL^F}(Q)$ such that $C_Q=\Tr^{\N_{\bL^F}(Q)}_{\N_{\bL^F}(Q,c_Q)}(c_Q)$.
Let $(Q,b_Q)=\Theta(Q,c_Q)$ and $B_Q=\Tr^{\N_{\bL^F}(Q)}_{\N_{\bL^F}(Q,c_Q)}(b_Q)$.

\begin{lem}\label{lem:inj-Nor-bl}
	The map $\Phi$ from $\bNBr(\bL^F,c)$ to $\bNBr(\bG^F,b)$ given by $(Q,C_Q)\mapsto (Q,B_Q)$  is injective and $(\tbL^F\cA)_c$-equivariant.
\end{lem}

\begin{proof}
	We first prove that $\Phi$ is well-defined and it suffices to show that for two $c$-Brauer pairs $(Q,c_Q)$ and $(Q,c_Q')$, 
	$c_Q$ and $c_Q'$ are $\N_{\bL^F}(Q)$-conjugate if and only if $b_Q$ and $b_Q'$ are $\N_{\bG^F}(Q)$-conjugate, where $(Q,b_Q)=\Theta(Q,c_Q)$ and  $(Q,b_Q')=\Theta(Q,c_Q')$.
	This follows by Remark \ref{remark:conj-Br-pairs-blocks}.	 
	By Lemma \ref{lem:equivariant-Br-pairs}, $\Phi$ is $\tbL^F\cA$-equivariant. 
\end{proof}

\begin{rmk}\label{remark:conj-pairs-blocks}
	Let $(Q,B_Q),(Q',B_{Q'})\in \Phi(\bNBr(\bL^F,c))$. 
	Then by Lemma \ref{lem:inj-Nor-bl} and Remark \ref{remark:conj-Br-pairs-blocks}, $(Q,B_Q)$ and $(Q',B_{Q'})$ are $\bG^F$-conjugate if and only if they are $\bL^F$-conjugate.
	In addition, $\Phi$ induces a bijection from the  $\bL^F$-conjugacy classes of $\bNBr(\bL^F,c)$ is equal to the $\bG^F$-conjugacy classes of $\bNBr(\bG^F,b)$ which is $(\tbL^F\cA)_c$-equivariant. 
\end{rmk}

\subsection{Morita equivalence for block algebras of quotient groups}

	Let $(Q,C_Q)\in \bNBr(\bL^F,c)$ and $(Q,B_Q)=\Phi(Q,C_Q)$.  Then by \cite[Thm.~3.10]{Ru20a}, we have

\begin{thm}[Ruhstorfer]\label{thm:equ-normaliser}
 The bimodule $H_c^{\dim}(\bY^{\N_{\bG}(Q)}_{\C_{\bU}(Q)},\Lambda)C_Q$  induces a Morita equivalence between $\Lambda \N_{\bL^F}(Q)C_Q$ and $\Lambda \N_{\bG^F}(Q)B_Q$.
\end{thm}

In order to consider the Mortia equivalence between blocks of $\N_{\bL^F}(Q)/Q$ and $\N_{\bG^F}(Q)/Q$, we first recall a technical result from \cite{Ro98}.

\begin{lem}\label{lem:quotient}
	Let $Q$ be a common normal $\ell$-subgroup of finite groups $G$ and $H$, and let $e$, $f$ be central idempotents of $G$, $H$ respectively. 
	Let $S$ be an $\ell$-subgroup of $G\ti H^{\opp}$ such that $S\cap(1\times H^{\opp})=S\cap (G\times 1)=1$ and $Q\times Q^{\opp}\le (Q\times 1)S=(1\times Q^{\opp})S$. 
		Suppose that $\cC$ is a bounded complex of $(\Lambda Ge\otimes(\Lambda Hf)^{\opp})$-modules, each of which is a direct sum of some  direct summands of $\Ind_{S}^{G\ti H^{\opp}}\Lambda_{S}$, where $\Lambda_{S}$ denotes the trivial $\Lambda S$-module.
	Let $\overline G=G/Q$, $\overline H=H/Q$ and let
	 $\bar{e}$ and $\bar{f}$ be the images of $e$ and $f$ through the canonical morphisms $\Lambda G\to \Lambda \overline{G}$ and $\Lambda H\to \Lambda \overline{H}$, and $\overline{\cC}=\Lambda\overline{G}\bar{e}\otimes_{\Lambda G}\cC\otimes_{\Lambda H} \Lambda \overline{H}\bar{f}$. 
	 
	 Then $\cC$ induces a Rickard equivalence between $\Lambda Ge$ and $\Lambda Hf$ if and only if $\overline{\cC}$ induces a Rickard equivalence between $\Lambda \overline{G}\bar{e}$ and $\Lambda \overline{H}\bar{f}$.
\end{lem}

\begin{proof}
	This follows from the proof of Lemma 10.2.11 of \cite{Ro98}.
	Note that the assumption that each indecomposable direct summand of 
	each component of $\cC$ has trivial source and vertex $S$ there is to ensure that each of these indecomposable direct summands is isomorphic to a
	indecomposable direct summand of $\Ind_{S}^{G\ti H^{\opp}}\Lambda_{S}$.
\end{proof}

Now we keep the hypotheses and setup of Theorem \ref{thm:equ-normaliser}.
Write $\overline{\N_{\bG^F}(Q)}=\N_{\bG^F}(Q)/Q$ and $\overline{\N_{\bL^F}(Q)}=\N_{\bL^F}(Q)/Q$.

\begin{lem}\label{lem:homo-gp}
	Let $\overline{\G\Ga_c(\bY_{\C_{\bU}(Q)}^{\N_{\bG}(Q)},\Lambda)}=\Lambda\overline{\N_{\bG^F}(Q)}\otimes_{\Lambda\N_{\bG^F}(Q)} \G\Ga_c(\bY_{\C_{\bU}(Q)}^{\N_{\bG}(Q)},\Lambda)\otimes_{\Lambda\N_{\bL^F}(Q)}\Lambda\overline{\N_{\bL^F}(Q)}$.	
	\begin{enumerate}[\rm(i)]
		\item $H^i(\overline{\G\Ga_c(\bY_{\C_{\bU}(Q)}^{\N_{\bG}(Q)},\Lambda)})\ne 0$ if and only if $i=\dim (\bY_{\C_{\bU}(Q)}^{\N_{\bG}(Q)})$.
		\item $H^{\dim( \bY_{\C_{\bU}(Q)}^{\N_{\bG}(Q)})}(\overline{\G\Ga_c(\bY_{\C_{\bU}(Q)}^{\N_{\bG}(Q)},\Lambda)})\simeq\overline{H^{\dim}(\bY_{\C_{\bU}(Q)}^{\N_{\bG}(Q)},\Lambda)}$.
	\end{enumerate}
\end{lem}

\begin{proof}
Let $\mathcal C_0=\G\Ga_c(\bY_{\C_{\bU}(Q)}^{\N_{\bG}(Q)},\Lambda)^{\mathrm{red}}$. 	Denote $G=\bG^F$ and $L=\bL^F$.
Since $\G\Ga_c(\bY_{\bU}^{\bG},\cO)^{\mathrm{red}}$ is splendid, each indecomposable direct summand of components of $\mathcal C_0$ is a direct summand of $\Ind^{\N_G(Q)}_{\Delta T}\Lambda_{\Delta T}$ where $Q\le T$.

First let $\Lambda=k$.
By the proof of Lemma 10.2.11 of \cite{Ro98}, we know
$$\mathcal C_0\otimes _{k \N_L(Q)}k \overline{\N_L(Q)}\simeq k \overline{\N_G(Q)}\times_{k \N_G(Q)} \mathcal C_0\otimes _{k \N_L(Q)} \overline{\N_L(Q)}:=\overline{\mathcal C}_0$$
in $\mathrm{Comp}^b(k (\N_G(Q)\ti\N_L(Q)^{\opp}))$.
Note that as $k \N_L(Q)$-modules, the components of $\mathcal C_0$ are projective.
Then $\mathcal C_0\simeq H^{\dim (\bY_{\C_{\bU}(Q)}^{\N_{\bG}(Q)})}(\mathcal C_0)$ in $\mathrm{Ho}^b(k\N_L(Q))$.
Thus $$\mathcal C_0\otimes _{k \N_L(Q)}k \overline{\N_L(Q)}\simeq H^{\dim (\bY_{\C_{\bU}(Q)}^{\N_{\bG}(Q)})}(\mathcal C_0) \otimes _{k \N_L(Q)}k \overline{\N_L(Q)}$$
in $\mathrm{Ho}^b(k\N_L(Q))$.
So $\overline{\mathcal C}_0\simeq H^{\dim (\bY_{\C_{\bU}(Q)}^{\N_{\bG}(Q)})}(\mathcal C_0) \otimes _{k \N_L(Q)}k \overline{\N_L(Q)}$
in $\mathrm{Ho}^b(k\N_L(Q))$, which means (i).

Now $$H^{\dim (\bY_{\C_{\bU}(Q)}^{\N_{\bG}(Q)})}(\overline{\mathcal C}_0)\simeq H^{\dim (\bY_{\C_{\bU}(Q)}^{\N_{\bG}(Q)})}(\mathcal C_0) \otimes _{k \N_L(Q)}k \overline{\N_L(Q)}$$
as $k (\N_G(Q)\ti\N_L(Q)^{\opp})$-modules.
Since $Q\ti 1$ acts trivially on $H^{\dim}(\overline{\mathcal C}_0)$, so does on $$H^{\dim (\bY_{\C_{\bU}(Q)}^{\N_{\bG}(Q)})}(\mathcal C_0) \otimes _{k \N_L(Q)} k \overline{\N_L(Q)}.$$
Thus 
\begin{align*}
H^{\dim (\bY_{\C_{\bU}(Q)}^{\N_{\bG}(Q)})}(\overline{\mathcal C}_0)\simeq& H^{\dim (\bY_{\C_{\bU}(Q)}^{\N_{\bG}(Q)})}(\mathcal C_0) \otimes _{k \N_L(Q)} k \overline{\N_L(Q)}\\
\simeq& k \overline{\N_G(Q)}\times_{k \N_G(Q)} H^{\dim (\bY_{\C_{\bU}(Q)}^{\N_{\bG}(Q)})}(\mathcal C_0)\otimes _{k \N_L(Q)} \overline{\N_L(Q)}\\
\simeq& \overline{H^{\dim (\bY_{\C_{\bU}(Q)}^{\N_{\bG}(Q)})}(\mathcal C_0)},
\end{align*}
which gives (ii).

Now we let $\Lambda=\cO$.
(i) is obvious.
We claim that $\mathcal C_0\simeq H^{\dim (\bY_{\C_{\bU}(Q)}^{\N_{\bG}(Q)})}(\mathcal C_0)$ in $\mathrm{Ho}^b(\Lambda\N_L(Q))$.
Assume that out claim does not hold.
Note that as $\cO \N_L(Q)$-modules, the components of $\mathcal C_0$ are projective.
So by \cite[Lemma 1.2]{Du12}, we have a decomposition
$H^{\dim (\bY_{\C_{\bU}(Q)}^{\N_{\bG}(Q)})}(\mathcal C_0)\simeq H_1\oplus H_2$ as $\cO$-modules where $H_1$ is $\cO$-free and $H_2\ne 0$ is $\cO$-torsion,
and $\mathcal C_0$ is homotopy equivalent to a complex
$$\cdots \to 0 \to M_1\to M_2\to 0\to \cdots$$
where the non-zero modules $M_1$ and $M_2$ are $\cO$-free.
Let $r_i$ be the $\cO$-rank of $M_i$ for $i=1,2$. 
Since the cohomology of $\mathcal C_0$ vanishes outside one degree, we know that the map $M_1\to M_2$ is injective.
Then the $\cO$-rank of $H_1$ is $r_2-r_1$.
On the other hand,  the cohomology of the complex
$$\cdots \to 0 \to M_1\otimes_\cO k\to M_2\otimes_\cO k\to 0\to \cdots$$
vanishes outside one degree,
since the cohomology of $\mathcal C_0\otimes_\cO k$ vanishes outside one degree. 
So $r_i$ is the $k$-rank of $M_i$ and then
$H^{\dim (\bY_{\C_{\bU}(Q)}^{\N_{\bG}(Q)})}(\mathcal C_0\otimes_\cO k)$ is of $k$-dimension $r_2-r_1$.
However, this contradicts $H_2\ne 0$.
Thus our claim holds.
Then (ii) holds by a similar argument as in the above paragraph.
Thus we complete the proof.
\end{proof}

\begin{thm}\label{thm:morita-quotient}
The bimodule $\overline{H_c^{\dim}(\bY^{\N_{\bG}(Q)}_{C_{\bU}(Q)},\Lambda)C_Q}$  induces a Morita equivalence between $\Lambda \overline{\N_{\bL^F}(Q)}\bar{C}_Q$ and $\Lambda \overline{\N_{\bG^F}(Q)}\bar{B}_Q$.
\end{thm}

\begin{proof}
We first claim that $\overline{\G\Ga_c(\bY^{\N_{\bG}(Q)}_{\C_{\bU}(Q)},k)C_Q}$  induces a Rickard equivalence between $k \overline{\N_{\bL^F}(Q)}\bar{C}_Q$ and $k \overline{\N_{\bG^F}(Q)}\bar{B}_Q$.	
According to Lemma \ref{lem:quotient}, Theorem \ref{thm:equ-normaliser} and \cite[Prop.~3.7]{Ru20a},
 it suffices to prove the following: 
	the indecomposable direct summands of all components of $$\Ind_{\N_{G\times L^{\opp}}(\Delta Q)}^{\N_G(Q)\times N_L(Q)^{\opp}}\G\Ga_c(\bY_{\C_{\bold U}(Q)}^{\C_{\bold G}(Q)},k)$$ have trivial source and vertices of the form $\Delta S$ where $S$ is an $\ell$-subgroup of $\N_L(Q)$ and $Q\subseteq S$. 
	Note that $$\Ind_{\N_{G\times L^{\opp}}(\Delta Q)}^{\N_G(Q)\times N_L(Q)^{\opp}}\G\Ga_c(\bY_{\C_{\bold U}(Q)}^{\C_{\bold G}(Q)},k)=\Ind_{\N_{G\times L^{\opp}}(\Delta Q)}^{\N_G(Q)\times \N_L(Q)^{\opp}}\Br_{\Delta Q}(\G\Ga_c(\bY_{\bold U}^{\bold G},k))$$ and $\G\Ga_c(\bY_{\bold U}^{\bold G},k)$ is splendid, 
	which means  indecomposable direct summands of its components have trivial source and vertices of the form $\Delta P$ where $P$ is an $\ell$-subgroup of $L$.
	Thus the indecomposable direct summands of  the components of
	$$\G\Ga_c(\bY_{\C_{\bold U}(Q)}^{\C_{\bold G}(Q)},k)=\Br_{\Delta Q}(\G\Ga_c(\bY_{\bold U}^{\bold G},k))$$ 
	have trivial source and vertices of the form $\Delta T$ where $T$ is an $\ell$-subgroup of $\N_L(Q)$ and $Q\subseteq T$ (actually, each component of $\G\Ga_c(\bY_{\C_{\bold U}(Q)}^{\C_{\bold G}(Q)},k)$ is a $k(\N_{G\times L^{\opp}}(\Delta Q)/\Delta{Q})$-module). 
Thus our claim holds.

Note that the complex $\G\Ga_c(\bY^{\N_{\bG}(Q)}_{\C_{\bU}(Q)},\cO)C_Q$ is a splendid complex of $\cO \N_G(Q)$-$\cO\N_L(Q)$-bimodules, which is a lift to $\cO$ of $\G\Ga_c(\bY^{\N_{\bG}(Q)}_{\C_{\bU}(Q)},k)C_Q$.
By \cite[Thm.~5.2]{Ri96} it follows that  $\overline{\G\Ga_c(\bY^{\N_{\bG}(Q)}_{\C_{\bU}(Q)},\cO)C_Q}$  induces a Rickard equivalence between $\cO \overline{\N_{\bL^F}(Q)}\bar{C}_Q$ and $\cO \overline{\N_{\bG^F}(Q)}\bar{B}_Q$.	

By Lemma \ref{lem:homo-gp}, 
	$H^{\dim( \bY_{\C_{\bU}(Q)}^{\N_{\bG}(Q)})}(\overline{\G\Ga_c(\bY_{\C_{\bU}(Q)}^{\N_{\bG}(Q)},\Lambda)})\simeq\overline{H^{\dim}(\bY_{\C_{\bU}(Q)}^{\N_{\bG}(Q)},\Lambda)}$ is the only non-trivial cohomology of 
	$\overline{\G\Ga_c(\bY_{\C_{\bU}(Q)}^{\N_{\bG}(Q)},\Lambda)}$.
	Thus the bimodule $\overline{H_c^{\dim}(\bY^{\N_{\bG}(Q)}_{C_{\bU}(Q)},\Lambda)C_Q}$  induces a Morita equivalence between $\Lambda \overline{\N_{\bL^F}(Q)}\bar{C}_Q$ and $\Lambda \overline{\N_{\bG^F}(Q)}\bar{B}_Q$.
\end{proof}

\subsection{Jordan decomposition of weights}

\begin{lem}\label{lem:nor-wei-sta}
	Let $(Q,C_Q)\in \bNBr(\bL^F,c)$ and $(Q,B_Q)=\Phi(Q,C_Q)$. 
	Then $\N_{\tbG^F\cA}(Q,B_Q)=\N_{\bG^F}(Q) \N_{\tbL^F \cA}(Q,C_Q)$.
\end{lem}

\begin{proof}
	By Lemma \ref{lem:inj-Nor-bl}, one has $\N_{\bG^F}(Q) \N_{\tbL^F \cA}(Q,C_Q)\subseteq \N_{\tbG^F\cA}(Q,B_Q)$.
	Let $gh\in \N_{\tbG^F\cA}(Q,B_Q)$,  where $g\in\bG^F$ and $h\in \tbL^F \cA$.
	Then $(Q,B_Q)^g=(Q,B_Q)^{h^{-1}}:=(Q',B_Q')$.
	By Remark	\ref{remark:conj-pairs-blocks}, there exists $l\in \bL^F$ such that $(Q,B_Q)^l=(Q',B_Q')$.
	In particular, $Q^l=Q^g$ and thus $g=xl$, where $x\in \N_{\bG^F}(Q)$.
	Also by Lemma \ref{lem:inj-Nor-bl},  $(Q,C_Q)^{h^{-1}}=(Q,C_Q)^l=(Q',C_Q')$ such that $(Q,B_Q')=\Phi(Q,C_Q')$.
	So $lh\in \N_{\tbL^F \cA}(Q,C_Q)$.
	We write $gh=xl'$ where $l'=lh$.
	Since $gh\in \N_{\tbG^F\cA}(Q,B_Q)$,
	one has $l'\in \N_{\tbL^F\cA}(Q,B_Q)=\N_{\tbL^F\cA}(Q,C_Q)$.
	Therefore, we have proved that $\N_{\tbG^F\cA}(Q,B_Q)=\N_{\bG^F}(Q) \N_{\tbL^F \cA}(Q,C_Q)$.	
\end{proof}

Denote by $R^{\N_{\bG}(Q)}_{\N_{\bL}(Q)}$ the bijection from $\Irr(C_Q)$ to $\Irr(B_Q)$ induced by the bimodule $H_c^{\dim}(\bY^{\N_{\bG}(Q)}_{\C_{\bU}(Q)},\Lambda)C_Q$ as in Theorem \ref{thm:equ-normaliser}.
It is just the Lusztig induction up to sign.
Now we can establish a Jordan decomposition of weights and hence an equivariant bijection between $\Alp(c)$ and $\Alp(b)$.

\begin{thm}\label{prop:mor-equ-quotient}
	\begin{enumerate}[\rm(i)]
		\item $\mathcal R^{\bG}_{\bL}:(Q,\vhi)\mapsto (Q, R^{\N_{\bG}(Q)}_{\N_{\bL}(Q)}(\vhi))$ gives an $(\tbL^F \mathcal A)_c$-equivariant injection from $\Alp^0(c)$ to $\Alp^0(b)$.
		\item 	This map $\mathcal R^{\bG}_{\bL}$ induces an $(\tbL^F \mathcal A)_c$-eqivariant bijection between $\Alp(c)$ and $\Alp(b)$.
		\end{enumerate}
\end{thm}

\begin{proof}
By Lemma  \ref{lem:inj-Nor-bl} and Theorem \ref{thm:morita-quotient}, $(Q,\vhi)\mapsto (Q,R^{\N_{\bG}(Q)}_{\N_{\bL}(Q)}(\vhi))$ gives an  injection between $\Alp^0(c)$ and $\Alp^0(b)$.
Now we show that this is $(\tbL^F \mathcal A)_c$-eqivariant.
Thanks to Lemma \ref{lem:inj-Nor-bl}, it suffices to show that for $(Q,C_Q)\in \bNBr(\bL^F,c)$ and $(Q,B_Q)=\Phi(Q,C_Q)$, the bijection
$$R^{\N_{\bG}(Q)}_{\N_{\bL}(Q)}: \dz(\N_{\bL^F}(Q)/Q)\cap \Irr(C_Q)\to \dz(\N_{\bG^F}(Q)/Q)\cap \Irr(B_Q)$$
is $\N_{\tbL^F \mathcal A}(Q,C_Q)$-eqivariant,
which follows by  Lemma \ref{lem:action-lie}.
By Remark \ref{remark:wei-fusion-system} and \ref{remark:Brauer-cat-splendid}, the map from (i) induces a bijection between $\Alp(c)$ to $\Alp(b)$ and this gives (ii).
\end{proof}

\begin{rmk}
	In \cite[\S5.2]{FM20}, a Jordan decomposition for weights of many classical groups was obtained	by using the combinatorial description of weights given by Alperin–Fong \cite{AF90} and An \cite{An94}.
	In the case when $\C_{\bG^*}(s)$ is a Levi subgroup of $\bG^*$, those can be also deduced from Theorem \ref{prop:mor-equ-quotient}.	
\end{rmk}

\begin{cor}\label{lem:nor-wei}
	Let $(Q,\vhi)\in \Alp^0(c)$ and $\vhi'=R^{\N_{\bG}(Q)}_{\N_{\bL}(Q)}(\vhi)$.
		Then $\N_{\tbG^F\cA}(Q)_{\vhi'}=\N_{\bG^F}(Q) \N_{\tbL^F \cA}(Q)_\vhi$.
\end{cor}

\begin{proof}
This follows by Lemma \ref{lem:nor-wei-sta} directly.
\end{proof}

Finally, we give the following result on the extendibility of weight characters.
Note that if $\tau:\bG\to\bG$ is an algebraic automorphism of $\bG$ of finite order,
then the semi-direct product $\bG\rtimes\langle \tau\rangle$ is a (not connected) reductive algebraic group; see \cite{Ma93}. 
Let  $B_Q'=\Tr_{\N_{\tbG^F\cA}(Q,B_Q)}^{\N_{\bG^F}(Q)\N_{\tbL^F\cA}(Q)}(B_Q)$, $C_Q'=\Tr_{\N_{\tbL^F\cA}(Q,B_Q)}^{\N_{\tbL^F\cA}(Q)}(C_Q)$.
Recall that $\bL$ and $e_s^{\bL^F}$ are assumed to be $\cA$-stable.

\begin{prop}\label{prop:ext-wei-char}
Let $A\in\{ K,\cO,k \}$. 
	Assume that  the order of $\sigma:\tbG^F\to \tbG^F$ is invertible in~$A$.
Then the bimodule $H_c^{\dim}(\bY^{\N_{\bG}(Q)}_{C_{\bU}(Q)},A)C_Q$ extends to an $A[\N_{\bG^F}(Q)\times \N_{\bL^F}(Q)^{\opp}\Delta(\N_{\tbL^F \cA}(Q,C_Q))]$-module $M_Q$ and the bimodule $\Ind_{\N_{\bL^F}(Q)\times \N_{\bL^F}(Q)^{\opp}\Delta(\N_{\tbL^F \cA}(Q,C_Q))}^{\N_{\bG^F}(Q) \N_{\tbG^F \cA}(Q)\times \N_{\tbL^F \cA}(Q)^{\opp}}(M_Q)$
induces a Morita equivalence between $A \N_{\bG^F}(Q)\N_{\tbL^F \cA}(Q)B_Q'$ and $A\N_{\tbL^F \cA}(Q)C_Q'$.	
In particular, if $(Q,\vhi)\in \Alp^0(c)$, $\vhi'=R^{\N_{\bG}(Q)}_{\N_{\bL}(Q)}(\vhi)$ and $\bL^F\le H\le \tbL^F \cA$, then $\vhi$ extends to $\N_H(Q)_\vhi$ if and only if $\vhi'$ extends to $\N_{\bG^F H}(Q)_{\vhi'}$.
\end{prop}

\begin{proof}
This can be proven by a similar argument as in \cite[Thm.~5.11]{Ru20a}.
That proof uses a theorem of Marcus \cite[Thm.~3.4]{Ma96} (we also refer to Theorem 1.7 and Lemma 1.11 of \cite{Ru20a}), which needs two conditions: $\N_{\tbG^F\cA}(Q,B_Q)=\N_{\bG^F}(Q) \N_{\tbL^F \cA}(Q,C_Q)$ and
 the extendibility of the $A\N_{\bG^F}(Q)$-$A\N_{\bL^F}(Q)$-bimodule $H_c^{\dim}(\bY^{\N_{\bG}(Q)}_{C_{\bU}(Q)},A)C_Q$.

We note that in \cite[Thm.~5.11]{Ru20a}, the group $Q$ is assumed to be a characteristic subgroup of a defect group  and this condition there is to ensure 
$\N_{\tbG^F\cA}(Q,B_Q)=\N_{\bG^F}(Q) \N_{\tbL^F \cA}(Q,C_Q)$, which follows by Lemma \ref{lem:nor-wei-sta} here.
The extendibility of $H_c^{\dim}(\bY^{\N_{\bG}(Q)}_{C_{\bU}(Q)},\Lambda)C_Q$ was shown in the proof of \cite[Thm.~5.11]{Ru20a}.
\end{proof}


\section{Reduction to quasi-isolated blocks}\label{sec:red-quasi-isolated}

We let $\bG$ be simple and of simply connected type such that $\bG/Z(\bG^F)$ is an abstract simple group in this section and
let $s\in{\bG^*}^F$ be a  semisimple $\ell'$-element which is not  strictly
quasi-isolated in $\bG^*$. 
Assume that $\bL^*$ is the minimal $F$-stable Levi subgroup of $\bG^*$ containing $\C^\circ_{{\bG^*}}(s)\C_{{\bG^*}}(s)^F$ .
Let $\iota:\bG\hookrightarrow\tbG$ be a regular embedding as in \cite[2.B]{MS16}.

If $(\bG,F)$ is not of type $\mathsf D_4$, then by \cite[Prop.~4.9]{Ru20a}, there exists a Frobenius endomorphism $F_0:\tbG\to\tbG$  and a bijective morphism $\sigma:\tbG\to\tbG$ such that $\cA:=\langle F_0,\sigma\rangle$ satisfies that
$F$ is a power of $F_0$,
$F_0\circ \sigma=\sigma\circ F_0$ as morphisms of $\tbG$,
and $\tbG^F\cA$ induces the stabilizer of $e_s^{\bG^F}$ in $\mathrm{Out}(\bG^F)$.
If $(\bG,F)$ is of type $\mathsf D_4$, then we let the group $\cA$ be as in \cite[\S2.3]{Ru20b}.
Note that $\bG^F\cA/\bG^F$ is abelian or isomorphic to the direct product of a
cyclic group and $\fS_3$.
In addition, there is a Levi subgroup $\bL$ of $\bG$ in duality to $\bL^*$ such that $\cA$ stabilizes $\bL$ and $e_s^{\bL^F}$.
Let $\tbL=\bL Z(\tbG)$.

\subsection{Going to Levi subgroups}

We abbreviate $\bG^F$, $\tbG^F$, $\bL^F$, $\tbL^F$ to $G$, $\tG$, $L$, $\tL$ respectively.

\begin{lem}\label{lem:bi-Br-ch}	
The bijection $R^{\bG}_{\bL}:\IBr(c)\to\IBr(b)$	is $(\tL\cA)_c$-equivariant and $R^{\tbG}_{\tbL}:\IBr(\tL\mid \IBr(c))\to \IBr(\tG\mid \IBr(b))$ is $\IBr(\tG/G)\rtimes \cA$-equivariant.
\end{lem}

\begin{proof}
This follows by  Lemma \ref{lem:action-lie}.	
Note that by
Lemma \ref{lem:equ-ind-mul},
$R^{\tbG}_{\tbL}$ is $\IBr(\tG/G)$-equivariant.
\end{proof}	

We denote by $\Alp(G,e_{s}^G)$ the union of the  $\Alp(b)$, where $b$ runs through the blocks of $G$ in $s$.

\begin{prop}\label{prop:iBAW-levi}
Assume that  the following hold.
\begin{enumerate}[\rm(i)]
	\item \begin{enumerate}[\rm(a)]
		\item In every $\tG$-orbit of $\IBr(G,e_s^G)$ there exists a Brauer character $\psi'\in\IBr(G,e_s^G)$ such that $(\tG \cA)_{\psi'}=\tG_{\psi'}\cA_{\psi'}$ and ${\psi'}$ extends to $G \cA_{\psi'}$.
		\item For any weight $(Q,\varphi)\in \Alp(L,e_s^L)$, the weight character $\varphi$ extends to $\N_{\tL}(Q)_\vhi$.
		\end{enumerate}
	\item There exists a blockwise $\Lin_{\ell'}(\tL/L)\rtimes \cA$-equivariant bijection $$\tilde f:\IBr(\tL\mid \IBr(L,e_s^L))\to \Alp(\tL\mid \Alp(L,e_s^L)).$$ 
	\item There exists a blockwise $(\tL \cA)$-equivariant bijection $f:\IBr(L,e_s^L)\to\Alp(L,e_s^L)$ such that
	\begin{enumerate}[\rm(a)]
		\item for any $\psi\in\IBr(L,e_s^L)$ and $f(\psi)=\overline{(Q,\vhi)}$, 
	if $\psi$ extends to a subgroup $H$ of $(L\cA)_\psi$, then $\vhi$ extends to $\N_H(Q)$,
			\item  $\tilde f(\IBr(\tL\mid\psi))=\Alp(\tL\mid f(\psi))$ for any $\psi\in\IBr(L,e_s^L)$.
	\end{enumerate}
	\item For every $\psi\in\IBr(c)$, $\tpsi\in\IBr(\tL\mid\psi)$ and 		$f(\psi)=\overline{(Q,\vhi)}$, $\tilde f(\tpsi)=\overline{(\tQ,\tvhi)}$	 the following hold:  $\bl(\hpsi)=\bl(\widehat\vhi)^{\tL_\psi}$, where 	 $\hpsi\in\IBr(\tL_\psi\mid\psi)$ is the Clifford correspondent of $\tpsi$ and	 $\widehat\vhi\in\Irr(\N_{\tL}(Q)_\vhi\mid\vhi)$ is the Clifford correspondent of $\Delta_\vhi^{-1}(\tpsi)$. 
\end{enumerate}
Then any block  $b$  in $s$ is BAW-good.	
\end{prop}

\begin{proof}
We will verify the criterion of the inductive BAW condition for $b$ via Theorem \ref{thm:criterion-block}.	
Let $\cB$ be the group defined in \cite[\S2.1]{Ru20b}.
Then the group $\tG$, $E=\cB$ satisfies Condition  \ref{thm:criterion-block} (i.a) and (i.b).
But for the other conditions of Theorem \ref{thm:criterion-block}, we can transfer from $\cB$ to $\cA$. 
This is entirely analogous with the arguments in \S 2.2--\S 2.5  of \cite{Ru20b},
which means for Condition  \ref{thm:criterion-block} (ii)--(iv), we may let $E=\cA$.
In fact, if $\chi\in \IBr(G)$, then $\cA_{\chi}$ and $\cB_\chi$ coincide as the subgroup of $\Out(G)$.
If moreover $\bG^F$ is not of type $\mathsf D_4$ and $\cA_\chi$ is not cyclic, then $\cA_{\chi}=\langle \mathrm{ad}(x)\gamma, F_0^i \rangle$
and $\cB_{\chi}=\langle \gamma, F_0^i \rangle$, 
where $\gamma$ is the graph automorphism of $\bG$ defined as in \cite[4.1]{Ru20a}, $x\in \bG^{F_0}$ and $\mathrm{ad}(x)$ denotes the inner automorphism of $G$ induced by $x$.
Similar as in the proof of \cite[Lemma 2.3]{Ru20b}, one can prove
$\chi$ extends to $G\cA_\chi$ if and only if it extends to $G\cB_\chi$.
A local version of such argument also holds, as \cite[Lemma 2.4]{Ru20b}.
We also note that when $\bG^F$ is of type $\mathsf D_4$,  similar results as in \cite[2.3]{Ru20b} also hold here.

By \cite[Thm.~B]{Ge93}, Condition  \ref{thm:criterion-block} (i.c) holds; see also \cite[Thm.~1.7.15]{GM20}.
Also, every Brauer character $\psi\in \IBr(L)$ extends to its stabilizer in $\tL_\psi$.
According to (i.b),  
for every weight $(Q,\vhi)$ in $\Alp^0(L,e_s^L)$, the weight character $\vhi$ extends to $\N_{\tL}(Q)_\vhi$.
By Proposition \ref{prop:ext-wei-char}, 
for every weight $(Q,\vhi')$ in $\Alp^0(G,e_s^G)$, the weight character $\vhi'$ extends to $\N_{\tG}(Q)_{\vhi'}$.
Thus Condition  \ref{thm:criterion-block} (i) holds.

By Lemma \ref{lem:bi-Br-ch}, the bijection
$R^{\bG}_{\bL}:\IBr(L,e_s^L)\to\IBr(G,e_s^G)$	is $(\tL\cA)$-equivariant and
by Theorem \ref{prop:mor-equ-quotient},
$\mathcal R^{\bG}_{\bL}$ gives a $(\tL \mathcal A)$-eqivariant bijection between $\Alp(L,e_s^L)$ and $\Alp(G,e_s^G)$.
We define $\Omega:\IBr(G,e_s^G)\to\Alp(G,e_s^G)$ to be the bijection which makes the following diagram commutative:
	\begin{align*}
	\xymatrix{
		&\IBr(L,e_s^L) \ar[r]^{R^{\bG}_{\bL}\ }\ar[d]_{f} & \IBr(G,e_s^G)\ar[d]_{\Omega}\\
		&\Alp(L,e_s^L)\ar[r]^{\mathcal R^{\bG}_{\bL}\ } & \Alp(G,e_s^G)}
\end{align*}
Then $\Omega$ is also $(\tL\cA)$-equivariant and then is $(\tG\cA)$-equivariant since $\tG\cA=G (\tL\cA)$.
So Condition  \ref{thm:criterion-block} (ii) holds.

By \cite[Prop.~1.1]{BR06}, one has the canonical isomorphism
$$\Ind^{\tG\ti \tL^{\opp}}_{(G\ti L^{\opp})\Delta \tL}\G\Ga_c(\bY_{\bU}^{\bG},\cO)e_s^{L} \simeq \G\Ga_c(\bY_{\bU}^{\tbG},\cO)e_s^{L}$$
in $\mathrm{Ho}^b(\cO(\tG\ti\tL^{\opp}))$.
According to \cite[Lemma 7.4]{BDR17}, the complex $\G\Ga_c(\bY_{\bU}^{\tbG},\cO)^{\mathrm{red}}e_s^{L}$
induces a splendid Rickard equivalence between $\cO\tL e_s^L$ and $\cO\tG e_s^G$.

Let $b$ be a block of $\bG^F$ in $s$.
If $\tilde b$ is a block of $\tbG$ covering $b$, then there exists a semisimple $\ell'$-element $\ts$ of $(\tbG^*)^F$ such that $\iota^*(\ts)=s$ and $\tilde b$ is in $\ts$.
Here, $\iota^*:\tbG^*\twoheadrightarrow \bG$ denotes the surjection induced by $\iota:\bG\hookrightarrow\tbG$.
In particular, $\C_{\tbG^*}(\ts)$ is contained in $\tbL^*$ according to \cite[2.A]{Bo05}. 
If $\tilde b'$ is another block of $\tG$ covering $b$, then $\tilde b'$ is in $\tilde z\ts$ for some $\ell'$-element $\tilde z\in Z((\tbG^*)^F)$.
Thus similarly as in the above paragraph, we can obtain a blockwise  bijection
\[\wOm: \IBr(\tG\mid \IBr(G,e_s^G))\to \Alp(\tG\mid \Alp(G,e_s^G))\] which is $\Lin_{\ell'}(\tG/G) \rtimes \cA$-equivariant by Lemma \ref{lem:equ-ind-mul} (ii) and\cite[Lemma~2.9~(b)]{Ru20b}.
In addition, by  Lemma \ref{lem:equ-ind-mul} (i), $\wOm(\IBr(\tG\mid \psi))=\Alp_\ell(\tG\mid \Omega(\psi))$ for every $\psi\in \IBr(G,e_s^G)$.
Therefore,  Condition  \ref{thm:criterion-block} (iii.a) and (iii.b) hold.
The proof of Condition  \ref{thm:criterion-block} (iii.c) is similar to \cite[Lemma 2.14]{Ru20b}.

By (i.a),   Condition  \ref{thm:criterion-block} (iv.a) holds.
Now we only need to verify Condition  \ref{thm:criterion-block} (iv.b).
Using (i.a), one has that  
in every $\tG$-orbit of $\IBr(G,e_s^G)$ there exists a Brauer character $\psi'\in\IBr(G,e_s^G)$ such that $(\tG \cA)_{\psi'}=\tG_{\psi'}\cA_{\psi'}$ and ${\psi'}$ extends to $G \cA_{\psi'}$.
Since $R^{\bG}_{\bL}$ is $(\tL \cA)$-equivariant, 
$(\tL \cA)_\psi=\tL_\psi \cA_\psi$
for $\psi\in\IBr(L,e_s^L)$ with $\psi'=R^{\bG}_{\bL}(\psi)$.
By \cite[Thm.~5.8]{Ru20b},  $\psi$ extends to $L \cA_\psi$.
Let  $\overline{(Q,\varphi')}=\Omega(\psi')$,
where $\psi'$ is as in the above paragraph.
First, $(\tG\cA)_{Q,\varphi'} = \tG_{Q,\varphi'} (G\cA)_{Q,\varphi'}$
since $\Omega$ is $(\tG\cA)$-equivariant.
Now $\psi$ extends to $L \cA_\psi$, by (ii.a), one has
$\vhi$ extends to $\N_{L\cA}(Q)_{\vhi}$.
By Proposition \ref{prop:ext-wei-char},  
 $\varphi'$ extends to $(G\rtimes \cA)_{Q,\varphi'}$.
 Note that if $\bG^F$ is of type $\mathsf D_4$, then a similar statement of \cite[Prop.~2.13]{Ru20b} for Brauer characters holds and the argument is similar with that of \cite[Thm.~2.12]{Ru20b}, by considering the extendibility of characters to Sylow 2- and 3-subgroups. 
 This complete the proof.
\end{proof}

In the spirit of Theorem \ref{thm:criterion-block} (iv.a) and
Proposition \ref{prop:iBAW-levi} (i.a), we make the following assumption:

\begin{amp}\label{assumption:stabilizer-extend}
	In every $\tbG^F$-orbit of $\IBr(\bG^F)$ there exists a Brauer character $\psi\in\IBr(\bG^F)$ such that $(\tbG^F \cA)_\psi=\tbG^F_\psi \cA_\psi$ and $\psi$ extends to $\bG^F \cA_\psi$.
\end{amp}

If  Assumption \ref{assumption:stabilizer-extend} is true then we say that it \emph{holds for $(\bG,F)$}.

\begin{rmk}\label{rmk:untri-basic-set}
	In the investigation of the inductive condition of McKay conjecture (cf. \cite[\S 10]{IMN07}), Cabanes and Sp\"ath proved in a series of papers \cite{CS17a,CS17b,CS19,Sp21a,Sp21b} for ordinary characters that:
	in every $\tG$-orbit of $\Irr(G)$ there exists a character $\chi\in\Irr(G)$ such that $(\tG \cA)_\chi=\tG_\chi \cA_\chi$ and $\chi$ extends to $G \cA_\chi$.
	They called this property the $A(\infty)$ condition.
	Thus Assumption \ref{assumption:stabilizer-extend} holds if $G$ has a $(\tG\cA)$-stable unitriangular basic set.
	
	Such unitriangular basic sets were constructed by Kleshchev--Tiep \cite{KT09} and  Denoncin \cite{De17} for $G=\SL_n(q)$ and $\SU_n(q)$, and then  Assumption \ref{assumption:stabilizer-extend} holds for $G$ and any semisimple $\ell'$-element $s$; see also \cite[Thm.~8.1]{FLZ20a}.
	When $s=1$,  the set of unipotent characters forms a unitriangular basic set when both $p$ and $\ell$ are good for $\bG$ by recent work of Brunat, Dudas and Taylor \cite{BDT20}, and then Assumption \ref{assumption:stabilizer-extend} is true for this case.
	For unipotent blocks of certain classical groups at the bad prime $\ell=2$, Chaneb \cite{Ch20} proved that there are also unitriangular basic sets; for example,  Assumption \ref{assumption:stabilizer-extend} holds for all irreducible Brauer characters of symplectic groups at the prime 2, see \cite[Cor.~4.6]{FM20}.
	For more results on	the unitriangular basic sets, see \cite[\S7]{FS21} and \cite[\S9.3.3]{Cr19}.
\end{rmk}

\subsection{Consider quasi-isolated blocks}

We consider the class $\mathcal H_{\bG}$ of pairs $(\bH,F')$ consisting of a simple algebraic group $\bH$ of simply connected type over $\overline{\F}_p$ with Steinberg map $F':\bH\to\bH$ such that
$\bH^{F'}/Z(\bH^{F'})$ is an abstract simple group and
 the Dynkin diagram of $\bH$ is isomorphic to a subgraph of the Dynkin diagram of $\bG$.

\begin{hpo}\label{Hypo-quasi-isolated}
Assume that for $(\bG,F)$ and a prime $\ell$ (different from $p$) the following holds: 
\begin{enumerate}[\rm(a)]
\item Assumption \ref{assumption:stabilizer-extend} holds for $(\bG,F)$, and 
\item  if  $(\bH,F')\in \mathcal H_{\bG}$ and 
$b$ is a  strictly 
quasi-isolated $\ell$-block of $\bH^{F'}$, then there exists an iBAW-bijection for $b$.
\end{enumerate}
\end{hpo}

Note that if the block $b$ is BAW-good, then Hypothesis \ref{Hypo-quasi-isolated}  (b) holds.
Recall that a semisimple element $s\in{\bH^*}^{F'}$ is  strictly quasi-isolated means $\C_{{\bH^*}^{F'}}(s)\C^\circ_{\bH^*}(s)$ is not contained in a proper Levi subgroup of $\bH^*$.

\begin{prop}\label{prop:ibaw-bij-c0}
	Keep the notation of Proposition \ref{prop:iBAW-levi} and assume that Hypothesis \ref{Hypo-quasi-isolated} (b)  holds for $(\bG,F)$ and $\ell$.
Let $c_0$ be a strictly 
quasi-isolated block of $L_0:=[\bL,\bL]^F$.
Then there exists an iBAW-bijection between $\IBr(c_0)$ and $\Alp(c_0)$.
\end{prop}

\begin{proof}
The proof is similar to \cite[Cor.~6.3]{NS14} and \cite[Prop.~3.7]{Ru20b}.
First, by \cite[Prop.~12.14]{MT11}, $[\bL,\bL]$ is a semisimple algebraic group and of simply connected type, \emph{i.e.},
$$[\bL,\bL]=\bH_1\ti \cdots\ti \bH_t$$
where the $\bH_i$'s are simple algebraic groups and of simply connected type.
Then $$[\bL,\bL]^*=\bH_1^* \ti \cdots \ti\bH_t^*$$ and $\bH^*_i$'s are simple groups of adjoint type.

Following the notation in the proof of \cite[Prop.~3.7]{Ru20b}, we let $\{ 1,2,\ldots,t\}=\Pi_1\cup  \cdots \cup\Pi_r$ be a partition with $|\Pi_i|=n_i$, such that $$L_0 =[\bL,\bL]^F\cong \bH_{x_1}^{F'_1}\ti \cdots\ti \bH_{x_r}^{F'_r},$$
where $F'_i=F^{n_i}$ and
 $x_i$ is a fixed element in $\Pi_i$
for every $1\le i\le r$.
Note that for $1\le i\le r$, the inclusion $\bH_{x_i}\hookrightarrow \prod_{j\in \Pi_{x_i}}\bH_j$ induces the isomorphisms between the finite groups $\bH_{x_i}^{F_i'}$ and $( \prod_{j\in \Pi_{x_i}}\bH_j)^F$. See the proof of \cite[Prop.~3.7]{Ru20b} for details.
Also   $$[\bL^*,\bL^*]^F\cong {\bH^*_{x_1}}^{F_1'}\ti \cdots\ti {\bH^*_{x_r}}^{F_r'}.$$

Denote $H_{x_i}=\bH_{x_i}^{F'_i}$ for $1\le i\le r$.
Let $s_0\in [\bL^*,\bL^*]^F$ be a semisimple $\ell'$-element such that $c_0$ is in $s_0$.
So $s_0$ is strictly 
quasi-isolated.
Write $s_0=s_{0,1}\ti \cdots\ti s_{0,r}$ with $s_{0,i}\in  {\bH^*_{x_i}}^{F_i'}$ and then 
$e_{s_0}^{L_0}=e_{s_{0,1}}^{H_{x_1}}\otimes\cdots\otimes e_{s_{0,r}}^{H_{x_r}}.$
The semisimple element $x_{0,i}$ is strictly 
quasi-isolated in $\bH^*_{x_i}$.
Also, the block $c_0$ can be write as $c_0=c_{x_1}\otimes \cdots\otimes c_{x_r}$ where $c_{x_i}$ is a block of $H_{x_i}$ so that $c_{x_i}$ is  strictly 
quasi-isolated.

By possibly reordering the factors of $L_0$ we assume that there exists some positive integer $v\le r$ such that the factor $H_{x_i}$ is quasi-simple if $i\le v$, while $H_{x_i}$ is almost simple or solvable if $i>v$.
Let $L_0=L_{0,\rm{simp}}\ti L_{0,\rm{solv}}$,
where $L_{0,\rm{simp}}\cong\prod_{i\le v} H_{x_i}$
and $L_{0,\rm{solv}}\cong\prod_{i>v} H_{x_i}$.
By a theorem of Tits \cite[Thm.~24.17]{MT11},
$L_{0,\rm{solv}}$ is a direct product of some groups in the following list:
$\SL_2(2)$, $\SL_2(3)$, $\SU_3(2)$, $\Sp_4(2)\cong \fS_6$, $G_2(2)$, $^2B(2)$, $^2G(3)$, $^2F_4(2)$.
Note that the inductive BAW condition is in fact established for these groups in \cite{FLZ20b,Ma14,Sch16} and then there exist strong iBAW-bijections for all of these groups.
Together with Hypothesis \ref{Hypo-quasi-isolated} (b), for every $1\le i\le r$, there exists an iBAW-bijection $h_i:\IBr(c_{x_i})\to\Alp(c_{x_i})$.

We note that $\IBr(c_0)=\IBr(c_{x_1})\ti\cdots\ti \IBr(c_{x_r})$.
On the other hand, for every $c_0$-weight $(Q_0,\vhi_0)$ of $L_0$, we have a decomposition $Q_0=Q_{x_1}\ti\cdots\ti Q_{x_r}$ with $Q_{x_i}\le H_{x_i}$.
Then $$N_{L_0}(Q_0)=N_{H_{x_1}}(Q_{x_1})\ti\cdots\ti N_{H_{x_r}}(Q_{x_r})$$ and $\vhi_0=\vhi_{x_1}\ti\cdots\ti \vhi_{x_r}$, where $\vhi_{x_i}\in\dz(N_{H_{x_i}}(Q_{x_i})/Q_{x_i})$.
Thus $(Q_{x_i},\vhi_{x_i})$ is a $c_{x_i}$-weight of $H_{x_i}$.
We denote $$(Q_0,\vhi_0):=(Q_{x_1},\vhi_{x_1})\ti\cdots\ti (Q_{x_r},\vhi_{x_r}).$$
Define the map
$f_0:\IBr(c_0)\to\Alp(c_0)$ by 
$$\psi_{x_1}\ti\cdots\ti \psi_{x_r}\mapsto h_1(\psi_{x_1})\ti\cdots\ti h_r(\psi_{x_r}),$$
where $\psi_{x_i}\in\IBr(c_{x_i})$.
Then $f_0$ is a bijection and we will prove that this is an iBAW-bijection.

Let $\{ x_1\ldots,x_r\}=A_1\cup\cdots\cup A_u$ be the partition such that $x_j,x_k\in A_i$ if and only if $n_j=n_k$  and there exists a bijective morphism $\phi:\bH_{x_j}\to\bH_{x_k}$ commuting with the action of $F'_i$ such that $\phi(c_{x_j})=c_{x_k}$.
For each $i$ we fix a representative $x_{i_j}\in A_i$.
Denote $y_i=x_{i_j}$.
Then without loss of generality, we may assume that $$c_0=\bigotimes_{i=1}^u c_{y_i}^{\otimes |A_i|},$$
where the $c_{y_i}$ are distinct blocks.
Therefore,
$$\Aut(L_0)_{c_0}\cong \prod_{i=1}^u \Aut(H_{y_i})_{c_{y_i}}\wr \fS_{|A_i|},$$
where $\fS_{|A_i|}$ is the symmetric groups on $|A_i|$ elements.
Then it can be checked directly that $f_0$ is $\Aut(L_0)_{c_0}$-equivariant.

Let $\psi_0:=\psi_{x_1}\ti\cdots\ti\psi_{x_r}$.
For $1\le i\le u$, we define a partition  $A_i=E_{i,1}\cup\cdots\cup E_{i,w_i}$ such that for $x_k,x_l\in A_i$, we have
$x_k,x_l\in E_{i,j}$ if and only if $\psi_{x_k}=\psi_{x_l}$.
For each pair $(i,j)$ we fix a representative $z_{i,j}\in E_{ij}$.
Then
$$\Aut(L_0)_{\psi_0}\cong \prod_{i=1}^u \prod_{j=1}^{w_i} \Aut(H_{z_{i,j}})_{\psi_{z_{i,j}}}\wr \fS_{|E_{i,j}|}.$$
So the stabilizer of $\psi$ is a direct product of wreath products.
We note that the relation $\geqslant_b$ is compatible with direct products and wreath products; see Theorem 2.18, 2.21 and 4.6 of \cite{Sp18} for such results for $\geqslant_b$ of character triples and we note that character triples can be replaced by modular character triples in those theorems without additional requirements. 
From this, $f_0$ is an iBAW-bijection and this proves this assertion.
\end{proof}	

\subsection{Jordan decomposition for the inductive BAW condition}

Now we can prove our main theorem.

\begin{thm}\label{main-thm-quasi-isolated}
Let $\bG$ be a simple algebraic group of simply connected type with a Steinberg map $F:\bG\to\bG$  such that $\bG^F/Z(\bG^F)$ is simple and $\bG^F$ is its universal covering.
Let $\ell$ be a prime not dividing $q$.
If Hypothesis \ref{Hypo-quasi-isolated} holds for $(\bG,F)$ and $\ell$, then every  $\ell$-block of $\bG^F$ is BAW-good.
\end{thm}

\begin{proof}
Let $b$ be a block of $G=\bG^F$ in a  semisimple $\ell'$-element $s$ of ${\bG^*}^F$.	
We first assume that~$s$ is not strictly quasi-isolated.
Keep the notation of Proposition~\ref{prop:iBAW-levi}.
Let $\bL^*$ be a minimal Levi subgroup of $\bG^*$ containing $\C_{\bG^*}^\circ(s)\C_{\bG^*}(s)^F$, we know that $s$ is a strictly 
quasi-isolated element of $\bL^F$.		
We verify the conditions of Proposition \ref{prop:iBAW-levi}.

Let $c$ be a block of $L$ in $s$ and
$c_0$ be a block of $L_0=[\bL,\bL]^F$ covered by $c$. 
Let $s_0$ be a semisimple $\ell'$-element of ${[\bL,\bL]^*}^F$ such that $c_0$ is in $s_0$.
Then by \cite[Lemma~3.6]{Ru20b} or \cite[Prop.~2.3]{Bo05}, $c_0$ is a strictly 
quasi-isolated block of $L_0$.
Note that $\tL\cA$ acts on $L_0$ and
the automorphisms induced by $\tL$ on $L_0$ are diagonal automorphisms.
By the hypothesis, Condition \ref{prop:iBAW-levi} (i.a) holds.

According to Proposition \ref{prop:ibaw-bij-c0}, there exists an iBAW-bijection $\IBr(c_0)\to\Alp(c_0)$.
Then there exists an $(\tL\cA)_{e_{s_0}^{L_0}}$-equivariant bijection $f_0:\IBr(L_0,e_{s_0}^{L_0})\to\Alp(L_0,e_{s_0}^{L_0})$
and by the Butterfly Theorem \ref{thm:butt-thm},
$$((\tL\cA)_{\psi_0},L_0,\psi_0)\geqslant_b (\N_{\tL\cA}(Q_0)_{\vhi_0},\N_{L_0}(Q_0),(\vhi_0)^0)$$
for every $\psi_0\in \IBr(L_0,e_{s_0}^{L_0})$ and $\overline{(Q_0,\vhi_0)}=f_0(\psi_0)$.
By \cite[Thm.~1.7.15]{GM20},
  $\psi_0$ extends to $\tL_{\psi_0}$ for every $\psi_0\in\IBr(L_0)$.
So for every weight $(Q_0,\vhi_0)$ of $\Alp^0(L_0,e_{s_0}^{L_0})$,
$\vhi_0$ extends to $\N_{\tL}(Q_0)_{\vhi_0}$. 

By Theorem \ref{cor:ext-wei-dgn} (ii), there exist a blockwise $\IBr(\tL/L_0)\rtimes \cA$-equivariant bijection $$\tilde f:\IBr(\tL\mid \IBr(L_0,e_{s_0}^{L_0}))\to \Alp(\tL\mid \Alp(L_0,e_{s_0}^{L_0}))$$
and a  blockwise $(\tL \cA)$-equivariant bijection $$f:\IBr(L\mid \IBr(L_0,e_{s_0}^{L_0}))\to\Alp(L\mid \Alp(L_0,e_{s_0}^{L_0})).$$
By the construction of $f$ and $\tilde f$ in the proof of Theorem \ref{cor:ext-wei-dgn}, Condition \ref{prop:iBAW-levi} (iv) holds and
$\tilde f(\IBr(\tL\mid\psi_0))=\Alp(\tL\mid f_0(\psi_0))$,
$f(\IBr(L\mid\psi_0))=\Alp(L\mid f_0(\psi_0))$ for any $\psi_0\in\IBr(L_0,e_{s_0}^{L_0})$.
Then by Lemma \ref{lem:cover-weight},
$\tilde f(\IBr(\tL\mid\psi))=\Alp(\tL\mid f(\psi))$ for any $\psi\in\IBr(L,e_s^L)$.

Therefore,  it suffices to verify (i.b) and (iii.a)  of Condition \ref{prop:iBAW-levi}.
Since  $\psi$ extends to $\tL_{\psi}$ for every $\psi\in\IBr(L)$  (using \cite[Thm.~1.7.15]{GM20}),
by Theorem \ref{cor:ext-wei-dgn} (ii), we know that $\vhi$ extends to $\N_{\tL}(Q)_{\vhi}$
for $\overline{(Q, \vhi)}=f(\psi)$.
So Condition \ref{prop:iBAW-levi} (i.b) holds.
Finally,  Condition \ref{prop:iBAW-levi} (iii.a) follows by Theorem \ref{cor:ext-wei-dgn} (ii).
Therefore, all conditions of Proposition \ref {prop:iBAW-levi} are satisfied and then $b$ is BAW-good.

Now assume that $s$ is strictly quasi-isolated.
Since Assumption \ref{assumption:stabilizer-extend}  holds for $(\bG,F)$, using Theorem~\ref{thm:criterion-block} we can obtain from the above arguments that $b$ is BAW-good if and only if there exists an iBAW-bijection for $b$.
\end{proof}	

Thus we have completed the proof of Theorem \ref{mian-thm}.
	
\begin{rmk}
	Let $S=\PSp_{2n}(q)$ (with odd $q$) be a simple group and $G=\Sp_{2n}(q)$.
In the proof of the main theorem of \cite{FM20}, the verification of the inductive BAW condition for $S$ at the prime~2 was reduced to the unipotent block of $\C_{G^*}(s)^*$. 
We note that 
\cite[Prop.~4.5]{FM20} can be deduced by Lemma \ref{lem:bi-Br-ch}, while
\cite[Prop.~5.6]{FM20} can be deduced by Theorem \ref{prop:mor-equ-quotient}.
In particular, the reduction to the unipotent 2-block of $\C_{G^*}(s)^*$ can be also deduced from Proposition \ref{prop:iBAW-levi}.

The reduction  to unipotent blocks there 
is slightly different from a direct application of Theorem \ref{main-thm-quasi-isolated} since in \cite{FM20} the groups $\SL_n(q)$ and $\SU_n(q)$ are not considered.
However, if we use \cite[Thm.~1]{FLZ21}, it suffices to consider the quasi-isolated 2-blocks of $\Sp_{2n}(q)$, which is the (unique) unipotent 2-block. 
\end{rmk}

\begin{rmk}\label{rmk:checked-case-uni}
There is an important class of quasi-isolated blocks: the unipotent blocks, which are better understood by the work of Cabanes--Enguehard \cite{CE94} and Enguehard \cite{En00}.
It seems easier to check the inductive BAW condition for  unipotent blocks.
Besides groups of types $\mathsf G_2$ and ${}^3\mathsf D_4$ (cf. \cite{Sch16}), and the Suzuki and Ree groups (cf. \cite{Ma14}), so far the inductive BAW condition has been checked for various cases on unipotent blocks:
Let $(G,q,\ell)$ be as follows with $\ell\nmid q$, then every unipotent $\ell$-block of $G$ is BAW-good for:
\begin{itemize}
	\item $G=\SL_n(q)$, $\SU_n(q)$ (cf. \cite{FLZ20b}), or
 $\Sp_{2n}(q)$ (cf. \cite{FLZ19,FM20}),
	\item $G=\Spin_{2n+1}(q)$ (with $n\ge 2$), $\Spin_{2n}^-(q)$ (with $n\ge 4$) or $\Spin_{2n}^+(q)$ (with $n\ge 5$) and both $q$ and $\ell$ are odd (cf. \cite{FLZ19}),
	\item $G=\mathsf E_6^\pm (q)$ (with $2,3\nmid q$) and $\ell\ge 5$ (cf. \cite{DLZ21}).
\end{itemize}

We mention that for the case $G=\Sp_{2n}(q)$ and odd $\ell$, the statement in \cite{FLZ19} assumes that $q$ is odd.
For even $q$, the proof for odd $q$ in \cite{FLZ19} also applies, using the weights given in \cite[\S6]{Li21}.
For the case $G=\Spin_{2n}^+(q)$ ($n\ge 5$), the assumption that $\ell$ is linear in \cite{FLZ19} is to ensure a unitriangular basic set, which is known by \cite{BDT20} now.
\end{rmk}



\end{document}